\documentclass[english,12pt]{scrartcl}
\usepackage[T1]{fontenc}
\usepackage[latin9]{inputenc}
\usepackage{a4wide}%{geometry}
\usepackage{babel}
\usepackage{prettyref}
\usepackage{mathtools}
\usepackage{amsmath}
\usepackage{amsthm}
\usepackage{amssymb}
\usepackage[unicode=true,pdfusetitle,
 bookmarks=true,bookmarksnumbered=false,bookmarksopen=false,
 breaklinks=false,pdfborder={0 0 1},backref=false,colorlinks=false]
 {hyperref}
\usepackage{colortbl}

\makeatletter
\theoremstyle{plain}
\newtheorem{thm}{\protect\theoremname}[section]
\theoremstyle{definition}

\theoremstyle{plain}
\newtheorem{lem}[thm]{\protect\lemmaname}
\theoremstyle{remark}
\newtheorem{rem}[thm]{\protect\remarkname}
\theoremstyle{plain}
\newtheorem{prop}[thm]{\protect\propositionname}
\theoremstyle{plain}
\newtheorem{cor}[thm]{\protect\corollaryname}
\theoremstyle{definition}
\newtheorem{example}[thm]{\protect\examplename}
\theoremstyle{definition}

\@ifundefined{date}{}{\date{}}
\usepackage{prettyref}
\usepackage{dsfont}
\usepackage{enumerate}
\allowdisplaybreaks

\renewcommand{\ae}{{\textnormal{a.e.}}}
\newcommand{\lin}{\operatorname{lin}}
\newcommand{\ev}{\operatorname{ev}}
\newcommand{\id}{\operatorname{id}}
\newcommand*{\e}{\mathrm{e}}

\newcommand{\R}{\mathbb{R}}
\newcommand{\1}{\mathds{1}}
\newcommand{\N}{\mathbb{N}}
\newcommand{\K}{\mathbb{K}}
\newcommand{\dom}{\operatorname{dom}}

\renewcommand{\d}{\,\mathrm{d}}
\renewcommand{\Re}{\operatorname{Re}}

\renewcommand{\tilde}{\widetilde}

\newrefformat{subsec}{Subsection \ref{#1}}
\newrefformat{prob}{Problem \ref{#1}}
\newrefformat{prop}{Proposition \ref{#1}}
\newrefformat{lem}{Lemma \ref{#1}}
\newrefformat{thm}{Theorem \ref{#1}}
\newrefformat{cor}{Corollary \ref{#1}}
\newrefformat{rem}{Remark \ref{#1}}
\newrefformat{exa}{Example \ref{#1}}
\newrefformat{sub}{Subsection \ref{#1}}
\newrefformat{eq}{(\ref{#1})}

\theoremstyle{definition}

\newrefformat{hyp}{Hypotheses \ref{#1}}

\makeatother

\providecommand{\corollaryname}{Corollary}
\providecommand{\definitionname}{Definition}
\providecommand{\examplename}{Example}
\providecommand{\lemmaname}{Lemma}
\providecommand{\problemname}{Problem}
\providecommand{\propositionname}{Proposition}
\providecommand{\remarkname}{Remark}
\providecommand{\theoremname}{Theorem}

\begin{document}
\title{State-dependent Delay Differential Equations on $H^1$}
\author{Johanna Frohberg\thanks{Institut f\"ur Angewandte Analysis, TU Bergakademie Freiberg, Germany}\ \ and
Marcus Waurick\thanks{Corresponding Author, Institut f\"ur Angewandte Analysis, TU Bergakademie Freiberg, Germany}}
\maketitle
\begin{abstract}\textbf{Abstract} Classically, solution theories for state-dependent delay equations are developed in spaces of continuous or continuously differentiable functions. The former can be technically challenging to apply in as much as suitably Lipschitz continuous extensions of mappings onto the space of continuous functions are required; whereas the latter approach leads to restrictions on the class of initial pre-histories. Here, we establish a solution theory for state-dependent delay equations for arbitrary Lipschitz continuous pre-histories and suitably Lipschitz continuous right-hand sides on the Sobolev space $H^1$. The provided solution theory is independent of previous ones and is based on the contraction mapping principle on exponentially weighted spaces. In particular, initial pre-histories are not required to belong to solution manifolds and the generality of the approach permits the consideration of a large class of functional differential equations even for which the continuity of the right-hand side has constraints on the derivative.\end{abstract}
\textbf{Keywords} State-dependent delay equations, functional differential equation, weak solutions\\
\label{=00005Cnoindent}\textbf{MSC2020 } Primary 34K05 Secondary 34K30
\newpage
\tableofcontents{}

\section*{Acknowledgements} 
We thank Prof Petr H\'ajek and Prof Friedrich Martin Schneider for providing some insights on (Lipschitz) retractions on the space of continuous functions. We thank Prof Schneider also for making us aware of Kirszbraun's Theorem. We thank Prof Hans-Otto Walther for drawing our attention to the work of Prof Junya Nishiguchi. We wish to cordially thank the anonymous referee for their detailed and thorough review of the original manuscript and its first revision helping to significantly improve the manuscript.\section{Introduction}

State-dependent delay equations are delay equations where the retarded argument of the right-hand side is dependent on the solution itself. This fact necessitates a different solution theory compared to other (delay) differential equations.

There are several approaches and notions to address unique existence of state-dependent delay equations in the literature. In order to sketch their main properties, we consider the following particular state-dependent delay equation. Note that in Section \ref{sec:sth}, we will provide a solution theory for Lipschitz continuous solutions for equations of the form \eqref{eq:sdd1}.

Given $h,T>0$ and a Lipschitz continuous pre-history $\phi\colon (-h,0)\to \R^n$, we consider
\begin{equation}\label{eq:sdd1}
   \begin{cases}
   x'(t) = g(t,x(t),x(t+r(x_{t}))),& t \in (0,T);\\
   x_{0} = \phi,
   \end{cases}
\end{equation}
where we used the usual notation $x_{t}(s)\coloneqq x(t+s)$;  $r$, taking values in $[-h,0]$, is a suitable delay functional defined on a set of functions defined on $[-h,0]$ and $g\colon [0,T]\times \R^n\times \R^n\to \R^n$ is continuous and Lipschitz continuous uniformly in $t$; that is, there exists $L\geq 0$ such that for all $t\in [0,T]$ and $u,v,w,z\in \R^n$, we have
\[
 \| g(t,u,w)-g(t,v,z)\|\leq L(\|u-v\|+\|w-z\|).
\]
If $r$ was constant, the equation at hand is a (state-independent) delay differential equation, which can be dealt with in a well-known straightforward manner. Indeed, integrating the equation, a fixed point of the integral equation is a solution of the differential equation. The contraction mapping principle can be employed either using a weighted norm or by a suitable time-stepping method analysing the Picard iterates, see, e.g., \cite{Drakalauch,PTW14}, \cite[Chapter 4]{STW22}.

The theory is somewhat less elementary if one considers non-constant $r$. In this situation, one needs to involve also the derivative of $x$: a counter-example for uniqueness in spaces of continuous functions is already given in \cite{W70}. We shall sketch a technical reason for the involvement of the derivative of the unknown in any fixed point statement. If one wants to use a similar approach as in the $r$-is-constant case, formal integration of \eqref{eq:sdd1} (assuming everything is well-defined) leads to 
\[
   x(t) =\phi(0)+ \int_0^t g(s,x(s),x(s+r(x_{s})))\d s.
\]
Then, a possible map $\Phi$ for any fixed point argument reads
\[
   \Phi \colon x\mapsto \phi(0)+ \int_0^t g(s,x(s),x(s+r(x_{s})))\d s.
\]
In order for $\Phi$ to be contractive on some function space and at the same time using Lipschitz continuity of $g$, one is required to estimate a term of the following form
\begin{multline*}
   \int_0^t \| x(s+r(x_{s})-y(s+r(y_{s})\|\d s \\ \leq \int_0^t\| x(s+r(x_{s})-y(s+r(x_{s})\|\d s+\int_0^t\| y(s+r(x_{s})-y(s+r(y_{s})\|\d s.
\end{multline*}
The first term on the right-hand side can easily be dealt with in several function spaces and associated norms, the second term requires knowledge of the size of the derivative of $y$ \emph{a priori}. Besides, note that in certain cases, the knowledge of the size of the derivative is also necessary for suitable Lipschitz estimates for $r$, see \cite[Section 3]{Walther} or Section \ref{sec:applpos} below.

There are several ways in the literature of addressing well-posedness of state-dependent delay equations. For a general introduction to delay and so-called functional differential equations, we refer to \cite{MC16,D95}. We note in passing that further generalisations of functional differential equations to infinite-dimensional state spaces are structurally more complicated, see, e.g., \cite{HJ10}. In any case, concerning state-dependent delay equations, we refer to the fundamental contributions \cite{HKWW06,Walther} for many examples and a guide to the literature. In order to position our approach to the ones in the literature, we briefly mention four main sources: 
\begin{enumerate}
\item[(i)] the solution manifold approach introduced by \cite{Walther},
\item[(ii)] the retraction approach introduced at least in \cite{MNP94} and used quite recently in \cite{BGR21}, 
\item[(iii)] the mixed norm approach in \cite{HT97}, and
\item[(iv)] the prolongation approach in \cite{N17,N18}.
\end{enumerate}
In (i) and (ii), the equation \eqref{eq:sdd1} is considered as functional differential equation. In \cite{Walther} solutions in the space of continuously differentiable functions are looked out for. As a consequence, the equation itself dictates a necessary condition for the initial value namely that $\phi'(0)=g(0,\phi(0),\phi(r(\phi)))$ (see also \cite{Walther2}). This condition defines a submanifold $X_g$ of $C^1$---the so-called solution manifold. It can be shown that given suitable differentiability assumptions on the right-hand side, well-posedness for the equation in question follows. {Moreover, in this situation, solutions starting in $X_g$ can be differentiated with respect to initial data (\cite{W04}) allowing  for a principle of linearised stability}. Ultimately, the well-posedness theorem is based on the contraction mapping principle in $C^1$ with a sensible analysis of the constants involved. In applications, {even though it is often only based on identifying the particular form of the derivative, to check} the individual conditions can be quite technical, see, e.g., \cite{GW14}. 

The approach in \cite{MNP94} is based on an existence result in \cite[Theorem 2.2.1]{HV93} (which in turn uses Schauder's fixed point theorem) and an a posteriori uniqueness analysis for `almost locally Lipschitz continuous functionals', i.e., functionals that are locally Lipschitz continuous with respect to the uniform norm if additionally the set of functions is constrained to be Lipschitz continuous. Hence, the underlying space for unique existence is the space of continuous functions subject to a Lipschitz condition. Identifying certain continuity conditions for the right-hand sides only on a specific set of continuous functions, the general theory is applied using certain (Lipschitz) retractions; that is, (Lipschitz) continuous mappings defined on the whole of $C$ projecting onto the said specific set of functions. Given that the underlying space is a space of continuous functions such retractions need to be considered specifically for each set in turn. We particularly refer to the convex and compact subset of $C$ constructed in \cite[Section 5]{HR23}, that is not a Lipschitz retract; see also Section \ref{sec:FDECW} for more technicalities arising in the retraction approach for continuous functions. In any case, using Lipschitz continuous functions as potential solutions only, one realises that the condition to lie on the solution manifold need not be met here (let alone that under the general assumptions considered in \cite{MNP94} the mere existence of the manifold is in question). The argument in \cite{BGR21} for existence of solutions for all times requires arguments involving positivity.

The approach in (iii) was developed in order to study differentiability in certain Sobolev type norms. Similar to the rationale developed here---in \cite{HT97}, a contraction mapping principle is applied namely a variation of the so-called Uniform Contraction Principle. The norms studied are mixed norm type spaces looking at the Lipschitz norm of the initial pre-history combined with a $W^{1,p}$-norm for the solution for positive times. Note that the contraction principle can be applied, if the final time horizon, $T$, is chosen small enough, see \cite[Lemma 4.2]{HT97}. 

In (iv), the author addresses well-posedness on a more fundamental level introducing the concept of prolongations. This leads to sufficient and necessary conditions for well-posedness, see \cite{N17}. Also the role of a potential jump in the continuity of the derivative is considered. In any case, however, the pre-history space is a space of continuous or continuously differentiable functions and the contraction mapping theorem is involved using a time-stepping approach for the integrated differential equation and metrics related to the uniform norm (possibly combined with its derivative), see \cite[p 3515]{N17} or \cite[Proof of Lemma 5.11]{N18}.

In this note we develop a solution theory of equations of the type in \eqref{eq:sdd1} {using} spaces of weakly differentiable functions; more precisely $H^1(-h,T;\R^n)$; that is, the space of square integrable functions with square integrable derivative. More so, we will provide a generalisation of \eqref{eq:sdd1} towards (non-autonomous) \emph{functional differential equations} on $H^1(-h,T;\R^n)$ {(i.e., where the right-hand side of \eqref{eq:sdd1} is not only given as \emph{delay differential equation} with some $g$ and some delay functional $r$, but where the right-hand side is assumed to be explicitly time-dependent and generally dependent on $x_{t}$ as $G$ below in Theorem \ref{thm:fdewpint})}. Similar to \cite{BGR21,MNP94}, we have to suitably restrict to Lipschitz continuous functions for the initial data. 

There are no other restrictions however and we obtain global in time solutions for \eqref{eq:sdd1} right away. The theorem containing the solution theory spelled out for functional differential equations reads as follows (in Theorem \ref{thm:wpfde} below we also state and prove a blow-up result). By $\|z\|_\infty$, we denote the essential supremum of a function $z$.

\begin{thm}\label{thm:fdewpint} Let $h,T>0$, $G\colon [0,T]\times H^1(-h,0;\R^n)\to\R^n$ continuous. Assume that for all $\beta>0$ there exists $L\geq 0$ such that for all $\phi,\psi \in H^1(-h,0;\R^n)$ with $\|\phi'\|_\infty,\|\psi'\|_\infty\leq \beta$ and $t\in [0,T]$
\[
   \|G(t,\phi)-G(t,\psi)\|\leq L\|\phi-\psi\|_{H^1(-h,0;\R^n)}.
\]
Then for all Lipschitz continuous $\phi\colon (-h,0)\to \R^n$ there exists $0<T_1\leq T$ and a unique solution $x\in H^1(-h,T_1;\R^n)$ of
\[
    \begin{cases}
      x'(t)=G(t,x_{t}),& (0<t<T_1),\\
      x_{(0)}=\phi.
    \end{cases}
\] It follows that $x$ is Lipschitz continuous.
\end{thm}
Already for \eqref{eq:sdd1}, by means of a counterexample we show that the space $H^1(-h,T;\R^n)$ for the initial data to lie in is too big warranting well-posedness. Our approach here is based on exponentially weighted spaces and is inspired by Morgenstern's proof for the classical Picard--Lindel\"of theorem using an exponentially weighted norm for the space of continuous functions. Exponentially weighted $L_2$-type spaces are natural in the context of partial differential equations (see \cite{Picard2009,STW22}) and can successfully be applied to ordinary (delay) differential equations (see \cite{Drakalauch,PTW14,STW22}). Note that in these references the time-derivative is established as a continuously invertible linear operator; in fact this is also some part of the reason why the present approach works. However, we shall suppress this viewpoint in the following in the interest of readability. In contrast to the approach (iii) (\cite{HT97}) from above, the Lipschitz continuity of the solution is not built in in the underlying space the contraction mapping principle is applied to, but rather a consequence of the regularity properties of the equation. Moreover, the exponential weights permit the consideration of the whole final time horizon right away and no time-stepping is needed. We note in passing that the mixed norm space in \cite{HT97} is no Hilbert space.

Using the Hilbert space $H^1(-h,T;\R^n)$ instead, we are in the position to make use of the projection theorem. This particularly proves useful when we employ retraction techniques to extend (Lipschitz) continuous functions on $H^1(-h,0;\R^n)$ as every closed convex subset  $C\subseteq H^1(-h,T;\R^n)$ is a retract of $H^1(-h,T;\R^n)$; the retraction is given by the metric projection onto $C$, $P_C$, see in particular Proposition \ref{prop:normproj} below. {This helps to show that} Theorem \ref{thm:fdewpint} also contains a solution theory {for right-hand sides initially given on certain subsets only.} Indeed, some particulars of {this} are presented in Proposition \ref{prop:aLc}. In this proposition, we show that the functional differential equations guaranteeing unique local existence as considered in \cite{MNP94,BGR21} provide a particular case of the solution theory spelled out here: Using $P_C$, we can always extend a suitably Lipschitz continuous functional to the whole of $H^1$ and then apply Theorem \ref{thm:fdewpint}. We emphasise that Kirszbraun's theorem, see \cite{K34}, even asserts that given a Lipschitz continuous $f\colon \dom(f)\to H_2$, $H_1,H_2$ Hilbert spaces, defined on \emph{any} subset $\dom(f)\subseteq H_1$ can be Lipschitz continuously extended to the whole of $H_1$; hinting at the flexibility allowed for in the present Hilbert space approach.

Concerning our proof strategy of Theorem \ref{thm:fdewpint}, Proposition \ref{prop:normproj} is in fact applied to $V_\beta\subseteq H^1(-h,T;\R^n)$ the set of weakly differentiable functions with derivative bounded by $\beta$. This will be used to identify the evaluation functional defined on $H^1(-h,0;\R^n)\times [-h,0]$ suitably composed with the delay functional $r$ as a Lipschitz continuous mapping on $V_\beta$ (see Theorem \ref{thm:sddcm}). The existence of solutions up to the final time $T$ for the more particular equations of the type \eqref{eq:sdd1} is then guaranteed with a contradiction argument and a blow-up technique. 

Another technical advantage of $H^1(-h,T;\R^n)$ (or rather $L_2$) is that the pre-history mapping assigning to each function $x$ the pre-histories $s\mapsto x_{s}$ can be realised as a Lipschitz continuous mapping with arbitrarily small norm given the considered norm is suitably (exponentially) weighted, see Theorem \ref{thm:prehnorm} or \cite{Drakalauch}. This observation already proved useful in the context of neutral (delay) differential equations, see \cite[Section 5.3.3]{Drakalauch}. Note that an analogous property for the space of continuous functions is not true, see \cite[Remark 4.15]{PTW14}. Since $H^1$ embeds continuously into the space of continuous functions, the presented approach lies somewhat inbetween the retraction approach in \cite{BGR21} and the $C^1$-technique in \cite{Walther}. Note that, however, similar to the results presented in \cite{BGR21}, the developed assumptions are easily applicable and more often than not boil down to routine arguments.

We remark that the solution theory developed here is independent of previous solution results in the literature and is based on the contraction mapping principle on $H^1$ combined with a retraction technique valid for Hilbert spaces. 
We emphasise that the present approach addresses the non-autonomous case right away and the assumptions are formulated in such a way that the well-known (non-autonomous) Picard--Lindel\"of theorem for ODEs is a simple corollary of the theory developed here. Indeed, classical Picard--Lindel\"of needs to be given an initial value only rather than an initial pre-history. Hence, the restriction to Lipschitz continuous pre-histories imposed in the well-posedness theorem here offers no constriction for the classical case. 

Next, we briefly summarise the outline of the paper. In Section \ref{sec:fap}, we gather some facts about weighted $L_2$- and $H^1$-spaces, particularly concerning the pre-history mapping. Section \ref{sec:keest} is concerned with the key estimate showing that the evaluation mapping (and thus the delay functional) can be suitably established as a Lipschitz continuous mapping. The solution theory for \eqref{eq:sdd1} on $H^1$ is provided in Section \ref{sec:sth}. The sharpness of the result is illustrated by an example similar to the one in \cite{W70}. We also provide a permanence principle that eventually yields a solution theory for when the delay functional $r$ is a priori not globally Lipschitz continuous on $H^1$. In Section \ref{sec:fde} one finds a suitable generalisation to functional differential equations giving more details and a proof for Theorem \ref{thm:fdewpint}. {Moreover, we treat a variant of \eqref{eq:sdd1} with several delays as an application of our abstract finding.} By means of concrete applications in Section \ref{sec:appl} we illustrate our theory with a classical example from \cite{Walther} and an equation from mathematical biology taken from \cite{GW14,BGR21}. In this section, we will particularly show how the assumptions on unique existence of solutions in \cite{BGR21} can be used to obtain a local solution theory also within the present situation. We close with a comments section, Section \ref{sec:FDECW}, and a conclusion in Section \ref{sec:con}.

\section{Functional analytic prerequisites and the pre-history map}\label{sec:fap}

We use standard terminology for the function spaces involved and briefly recall the most important ones here, see, also, e.g.~\cite{STW22}. For this, let $I\subseteq \R$ be a non-empty interval. For $\rho\in \R$,  $L_{2,\rho}(I)$ is the space of equivalence classes of {(real-valued)} square-integrable functions on $I$ with respect to the measure with Lebesgue density $t\mapsto \exp(-2\rho t)$; and $H_{\rho}^{1}(I)$ is {the subspace of $L_{2,\rho}(I)$ precisely containing those elements in $L_{2,\rho}(I)$ that are once weakly differentiable with derivative in $L_{2,\rho}(I)$. For this, let $L_{1,\textnormal{loc}}(I)$ be the space of equivalence classes of locally integrable functions. We call $f\in L_{1,\textnormal{loc}}(I)$  \textbf{weakly differentiable}} if there exists $g\in L_{1,\textnormal{loc}}(I)$ such that for all $\phi\in C_c^\infty(I)$, the space of infinitely often continuously differentiable functions with compact support in $I$, we have
\[
  - \int_I f(t)\phi'(t)\d  t =  \int_I g(t)\phi(t)\d  t,
\]
in this case one can show that $g\in L_{1,\textnormal{loc}}(I)$ is uniquely determined; we write $f' \coloneqq g$ and we set 
\begin{align*}
   H_{\textnormal{loc}}^{1}(I) &\coloneqq \{ f\in L_{2,\textnormal{loc}}(I);f\text{ weakly differentiable with } f'\in L_{2,\textnormal{loc}}(I)\}\text{ and }\\
    H_{\rho}^{1}(I)& \coloneqq \{ f\in H_{\textnormal{loc}}^{1}(I);f, f' \in L_{2,\rho}(I)\}.
\end{align*}
 In case of $\rho=0$, we also use $H^1(I)$ and $L_2(I)$, respectively. For $\rho\in \R$, we define
\[
  \langle x,y\rangle_{0,\rho}\coloneqq \int_{I}x(t)\cdot y(t)\exp(-2\rho t)dt\text{ and  }  \langle x,y\rangle_{1,\rho}\coloneqq \langle x,y\rangle_{0,\rho}+\langle x',y'\rangle_{0,\rho}
\]and, analogously, {we define the norms via $\|x\|_{L_{2,\rho}}\coloneqq \sqrt{\langle x,x\rangle_{0,\rho}}$ and $\|x\|_{H^1_{\rho}}\coloneqq\sqrt{\langle x,x\rangle_{1,\rho}}$}, respectively. {We mention that endowed with this scalar product $H_{\rho}^{1}(I)$ is a Hilbert space.} The corresponding Hilbert space of $X$-valued functions spaces will be denoted by $L_{2,\rho}(I;X)$,  etc.
 {Note that by the following version of the Sobolev embedding theorem, every $f\in H_{\rho}^{1}(I)$ {has a continuous representative}. \newline
\indent \emph{Throughout the manuscript, we shall almost always choose a continuous representative for such $f$ (which then is easily seen to be unique) allowing for pointwise arguments.}
\begin{thm}[Sobolev-embedding theorem]\label{thm:set} Let $I\subseteq \R$ be a non-empty interval. Then every $f\in H^1_{\textnormal{loc}}(I)$ admits a continuous representative {$\tilde{f}$}. If $f\in H^1(a,b)$ for some $-\infty<a<b<\infty$, then
\[
    \|\tilde{f}\|_{C[a,b]}\leq ((b-a)^{1/2}+(b-a)^{-1/2})\|f\|_{H^1(a,b)}.
\]
\end{thm}
\begin{proof}
The proof can be found in \cite[Theorem 4.9]{ISem18}; see also the standard reference \cite[Theorem 4.12, Case A]{AF03}.
\end{proof}
}
Throughout the remainder of this section, we let $h>0$ be our fixed pre-history horizon and $T \in (0,\infty]$ be the final time considered in this section. {In the next couple of lines, we want to define the pre-history map in the present $L_2$-type setting. For this recall that the mapping 
\[
C[-h,T] \ni f \mapsto( t\mapsto  f_t \in C[-h,0])\in C([0,T];C[-h,0])
\]
is well-defined.  This mapping can be extended to the $L_2$-setting in the following sense. For a function $f$ defined on an interval $I$, we denote by $[f]_{\ae}$ the equivalence class of Lebesgue measurable functions almost everywhere equal to $f$ on $I$. Recall that continuous functions with compact support in $I$, $C_c(I)$, are densely embedded into $L_2(I)$ in the sense that $C_c(I) \to L_2(I), f\mapsto [f]_\ae$ is one-to-one and has dense range. As a matter of jargon using a slight abuse of notation, for this dense embedding result, one usually just writes `$C_c(I)\subseteq L_2(I)$ densely'.
\begin{thm}\label{thm:prehnorm}
Let $\rho>0$. Then the mapping
\[
\tilde{\Theta}\colon C_c(-h,T)\subseteq L_{2,\rho}(-h,T)\to L_{2,\rho}(0,T,L_{2}(-h,0))
\]given by $\tilde{\Theta}f \coloneqq [(0,T)\ni t\mapsto  [f_t]_{\ae} ]_{\ae}$ is continuous and, thus, has a unique continuous extension denoted by $\Theta$ satisfying the norm estimate
\[
\|  \Theta\|\leq \frac{1}{\sqrt{2\rho}}.
\]
\end{thm}}
\begin{proof}
{ Let $f\in C_c(-h,T)$. Then we compute, using that the value of the Lebesgue integral does not depend on the individual Lebesgue measurable representative,}
 \begin{align*}
     &\|\tilde{\Theta} f\|^2_{L_{2,\rho}(0,T;L_{2}(-h,0))}\\& = \int_0^T \int_{-h}^0 |f(t+s)|^2 \d s \exp(-2\rho t)\d t 
      = \int_{-h}^0 \int_{0}^T |f(t+s)|^2  \exp(-2\rho t)\d t \d s 
  \\        & = \int_{-h}^0 \int_{0}^T |f(t+s)|^2  \exp(-2\rho( t+s))\d t \exp(2\rho s)\d s  \leq \int_{-h}^0  \exp(2\rho s)\d s \|[f]_\ae\|_{L_{2,\rho}(-h,T)}^2 \\
  & \leq \frac{1}{2\rho} \|[f]_\ae\|_{L_{2,\rho}(-h,T)}^2.
 \end{align*}
{ This estimate confirms continuity of $\tilde{\Theta}$. That densely defined continuous linear operators defined on a Banach space taking values in a possibly different Banach space can be uniquely extended to the whole of the first Banach space preserving its operator norm is a standard result, see, e.g., \cite[Corollary 2.1.5]{STW22}.}
\end{proof} 
We call the thus defined mapping $\Theta$ the \textbf{pre-history map}.
\begin{rem}\label{rem:thetaH}
 Note that $\Theta$ is linear and, thus, commutes with the weak derivative. %
 {Indeed, let $f\in H_{\rho}^1(-h,T)$ and choose its continuous representative. Then, for $\phi\in C_c^\infty(-h,0)$, which is extended by zero on the whole of $\R$, we compute for $t\in (0,T)$ using that {$(s\mapsto \phi(s-t))$ is supported in $(-h,T)$} and that $f$ is weakly differentiable  
 \begin{align*}
 -  \int_{-h}^0 ((\Theta f)(t))(s)\phi'(s)\d  s & = -  \int_{-h}^0 f(t+s)\phi'(s)\d  s\\
&  = -  \int_{-h+t}^t f(s)\phi'(s-t)\d  s =-  \int_{-h}^T f(s)\phi'(s-t)\d  s \\
& =  \int_{-h}^T f'(s)\phi(s-t)\d  s = \int_{-h+t}^t f'(s)\phi(s-t)\d  s  \\  & = \int_{-h}^0 ((\Theta f')(t))(s)\phi(s)\d  s,
 \end{align*}
 confirming $(\Theta f)(t)' = (\Theta f')(t)$, where all derivatives are weak derivatives.} 
  As a consequence, we have 
  { for $f\in H_{\rho}^1(-h,T)$}
  \begin{align*}
    \|\Theta f \|_{L_{2,\rho}(0,T,H^1(-h,0))}^2 &=    \int_{0}^T \|\Theta f(t) \|_{H^1(-h,0)}^2\exp(-2\rho t)\d t \\
    &=    \int_{0}^T \big(\|\Theta f(t) \|_{L_2(-h,0)}^2+\|(\Theta f(t))' \|_{L_2(-h,0)}^2\big)\exp(-2\rho t)\d t \\
      & =  \|\Theta f \|_{L_{2,\rho}(0,T,L_2(-h,0))}^2+\|t\mapsto ((\Theta f)(t)') \|_{L_{2,\rho}(0,T,L_2(-h,0))}^2 \\ &= \|\Theta f \|_{L_{2,\rho}(0,T,L_2(-h,0))}^2+\|\Theta (f') \|_{L_{2,\rho}(0,T,L_2(-h,0))}^2 \\
    &\leq  \frac{1}{2\rho}\big(\|f \|_{L_{2,\rho}(-h,T)}^2+\|f' \|_{L_{2,\rho}(-h,T)}^2 \big).
  \end{align*} Hence,   
  for
 \[
     \hat{\Theta}\colon H_{\rho}^1(-h,T)\to L_{2,\rho}(0,T,H^1(-h,0))
 \] given by $f\mapsto \Theta f$ we have
\[
\| \hat \Theta \|\leq \frac{1}{\sqrt{2\rho}}
\]\end{rem}
{For $f\in C[-h,T]$ and $t\in [0,T]$, we get $(s\mapsto f(t+s))\in C[-h,0]$. Thus, similar to the reasoning for the pre-history map, for $\rho\in \R$, $\tilde{\tau}_t \colon C_c[-h,T]\subseteq L_{2,\rho}(-h,T)\to L_2(-h,0)$ given by $\tilde{\tau}_t f \coloneqq [(s\mapsto f(t+s))]_\ae$ is linear and continuous. Thus, it admits a unique continuous extension, which we denote by $\tau_t$.}   Then a different representation of $\Theta$ is
\[
    \Theta f = (t\mapsto \tau_t f).
\]
{Indeed, the equality is readily confirmed for $f\in C_c(-h,T)$. Uniqueness of the continuous extension of densely defined bounded linear operators thus implies equality on the whole of $L_{2,\rho}(-h,T)$.}
This perspective helps us to prove the next observation concerning the pre-history map. {We introduce the space of $L_{2}(-h,0)$-valued continuous functions $C_\rho([0,T];L_{2}(-h,0))$ endowed with the norm given by $\|F\|_\rho\coloneqq \sup_{t\in [0,T]}\e^{-\rho t}\|F(t)\|_{L_2(-h,0)}$ for all $F\in C_\rho([0,T];L_{2}(-h,0))$.} 
\begin{prop}\label{prop:pre-h-cont}
 Let $\rho\in \R$. Then, for all $f\in L_{2,\rho}(-h,T)$, $\Theta f$ {has a unique representative in $C([0,T);L_{2}(-h,0)).$}
 If, in addition, $T<\infty$, then {$\Theta f$ has a unique representative in $C([0,T];L_{2}(-h,0))$ and $\Theta$ induces a map
 \[
    L_{2,\rho}(-h,T)\to C_\rho([0,T];L_{2}(-h,0)),
 \]which is continuous with operator norm bounded by $1$}.
\end{prop}
\begin{proof} The proof rests on a standard approximation technique and similar results can be found in (almost) any basic text on $C_0$-semi-group theory. We provide the short arguments here. First of all, since for all $t\in (0,\infty)$, we have $L_{2,\rho}(-h,\infty)\hookrightarrow L_{2,\rho}(-h,t), f\mapsto f|_{(-h,t)}$ continuously {the case $T=\infty$ readily follows from the case $T<\infty$. Hence, we may assume $T<\infty$ in the following.}

Let $f\in C[-h,T]$. Then $\tau_t f \to \tau_s f$ in $L_{2,\rho}(-h,0)$ as $t\to s$ in $[0,T]$ by continuity of $f$ and Lebesgue's dominated convergence theorem. 

Next, it is elementary to see that for each $t\in [0,T]$, we have
\[
   \|\tau_t f\|_{L_{2}(-h,0)}\leq \e^{\rho t}\|f\|_{L_{2,\rho}(-h,T)}.
\]
Thus, {dividing by $\e^{\rho t}$ and computing the supremum on both sides with respect to $t\in [0,T]$, we get
\[
 \|(t\mapsto \tau_t f)\|_{\rho} =  \sup_{t\in [0,T]} \e^{-\rho t}  \|\tau_t f\|_{L_{2}(-h,0)} \leq \|f\|_{L_{2,\rho}(-h,T)}.
\]}
Hence,
\[
    C[-h,T]\subseteq L_{2,\rho}(-h,T)\to C_\rho([0,T];L_{2}(-h,0)), f\mapsto \Theta f
\]
is continuous { as an operator from (a subset of) $L_{2,\rho}(-h,T)$ to $C_\rho([0,T];L_{2}(-h,0))$} with operator norm bounded by $1$. Continuous extension by density of $C[-h,T]$ in $L_{2,\rho}(-h,T)$ shows the remaining assertions.
\end{proof}

\begin{rem}\label{rem:conthatp}
Similar to Theorem \ref{thm:prehnorm}, we can also spell out the corresponding continuity statement for $H^1$ instead of $L_2$. More precisely, assuming $T<\infty$, it follows from Proposition \ref{prop:pre-h-cont} that for $f\in H^1_\rho(-h,T)$, {$\hat{\Theta}f$ has a continuous representative in $C_\rho([0,T];H^1(-h,0))$, which induces a continuous mapping with operator norm bounded by $1$. Indeed, using \cite[Theorem 7.7]{ISem18}, we recall that $C^1[-h,T]\subseteq H^1(-h,T)$ densely (in the sense of the embedding $f\mapsto [f]_\ae$ as before). With Proposition \ref{prop:pre-h-cont} we estimate for $f\in C^1[-h,T]$
\[
   \|t\mapsto [\Theta f(t)]_\ae\|_{C_\rho([0,T];L_{2}(-h,0))}\leq \|[f]_{\ae}\|_{ L_{2,\rho}(-h,T)}
\]
and
\[
   \|t\mapsto [\Theta f'(t)]_\ae\|_{C_\rho([0,T];L_{2}(-h,0))}\leq \|[f']_{\ae}\|_{ L_{2,\rho}(-h,T)}.
\]
By Remark \ref{rem:thetaH}, $\Theta f'(t)=\Theta f(t)'$ for all $t\in [0,T]$. For $C^1$-functions $f$, we have $[f']_\ae=[f]'_\ae$, where the former derivative is classical and the latter is weak. Thus, for $f\in C^1[-h,T]$, we obtain the estimate
\[
    \|t\mapsto [\Theta f(t)]_\ae\|_{C_\rho([0,T];H^1(-h,0))}\leq \|[f]_{\ae}\|_{ H^1_\rho(-h,T)},
\]
which implies that $\hat{\Theta}$ induces a continuous mapping with operator norm less than $1$ from $ H^1_\rho(-h,T)$ to $C_\rho([0,T];H^1(-h,0))$.}
In particular, for $\rho=0$ and for all $f\in H^1(-h,T)$, we have
\begin{equation}\label{eq:chf}
 \sup_{\tau \in [0,T]} \|f_\tau\|_{H^1(-h,0)}\leq \|f\|_{H^1(-h,T)}.
\end{equation}
\end{rem}

\section{A key estimate for the solution theory}\label{sec:keest}
 In order to provide a solution theory for state-dependent delay equations, we need to derive a key estimate. This estimate will render the anticipated functional differential equation as ordinary differential equation with Lipschitz continuous right-hand side. For this, we introduce the \textbf{Lipschitz semi-norm}, $\|g\|_{\textnormal{Lip}}$ for a mapping $g\colon X\to Y$ between metric spaces $(X,d)$ and $(Y,e)$ given by
 \[
    \|g\|_{\textnormal{Lip}} \coloneqq \sup_{x_1,x_2\in X,x_1\neq x_2} e(g(x_1),g(x_2))/d(x_1,x_2) \in [0,\infty].
 \] Evidently, $g$ is Lipschitz continuous if and only if $\|g\|_{\textnormal{Lip}}<\infty$.
 
For $\beta>0$, we introduce 
\[
  V_\beta \coloneqq \{ \psi \in H^1(-h,0;\R^n); \|\psi' \|_\infty \leq \beta\},
\]where $\|\cdot\|_\infty$ denotes the norm in $L_\infty(-h,0;\R^n)$ where $h>0$.

\begin{rem}\label{rem:Vbeta}
$V_\beta\subseteq H^1(-h,0;\R^n)$ is closed and  convex. Indeed, convexity being clear; we briefly address closedness. {Let $(\psi_k)_{k\in \N}$ be a sequence in $V_\beta$ converging to some $\psi\in H^1(-h,0;\R^n)$.} In particular, by the definition of the norm in  $H^1(-h,0;\R^n)$, $(\psi_k')_{k\in \N}$ converges to $\psi'$ in $L_2(-h,0;\R^n)$. Then, by the Fischer--Riesz theorem, we may assume without loss of generality that $(\psi_k')_{k\in \N}$ converges almost everywhere to $\psi'$; hence the $L_\infty$-norm estimate is preserved.
\end{rem}

In order to present the key ingredient warranting well-posedness of state-dependent delay equations, we reformulate said equation as a functional differential equation. For this, similarly to \cite{Walther}, we introduce the evaluation mapping $\ev\colon H^1(-h,0;\R^n) \times [-h,0] \to \R^n$ given by
\[
    \ev(\phi,t)\coloneqq \phi(t),
\]which is well-defined by Theorem \ref{thm:set}.
\begin{thm}\label{thm:sddcm}  Let $h, \beta>0$. Let both $g\colon \R^n\to \R^n$ and $r\colon H^1(-h,0;\R^n) \to [-h,0]$ be Lipschitz continuous.

Then $f\colon V_\beta\to \R^n$ with
\[
 f = g \circ \ev \circ (\id \times r)
\] satisfies
\[
   \|f\|_{\textnormal{Lip}} \leq \|g\|_{\textnormal{Lip}}(2h^{1/2}+\beta \|r\|_{\textnormal{Lip}} + h^{-1/2}).
\]
\end{thm}
\begin{proof} 
We compute for $\phi,\psi \in V_\beta$
  \begin{align*}
    \| f(\phi)-f(\psi)\|_{\R^n} & = \|g \circ \ev \circ (\id \times r)(\phi)-g \circ \ev \circ (\id \times r)(\psi)\|_{\R^n} \\
    & \leq \|g\|_{\textnormal{Lip}} \|\phi(r(\phi))-\psi(r(\psi))\|_{\R^n}. 
    \end{align*} 
    Focussing on the right-hand factor of the right-hand side, we further estimate
    \begin{align*}
    & \|\phi(r(\phi))-\psi(r(\psi))\|_{\R^n} \\
    & \leq\|\phi(0)-\phi(r(\phi)) +\psi(0)-\psi(r(\psi))\|_{\R^n}+\|\psi(0)-\phi(0)\|_{\R^n}\\
    & \leq\Big\|\int_{r(\phi)}^0 \phi'(s)\d s - \int_{r(\psi)}^0 \psi'(s)\d s\Big\|_{\R^n}+  (h^{1/2}+h^{-1/2})\|\psi-\phi\|_{H^1(-h,0;\R^n)},\end{align*}
    using Theorem \ref{thm:set}.
    Assuming without restriction that $r(\phi)\leq r(\psi)$ and using the Cauchy--Schwarz inequality, we treat the integral terms next:
    \begin{align*}
    \Big\|\int_{r(\phi)}^0 \phi'(s)\d s - \int_{r(\psi)}^0 \psi'(s)\d s\Big\|  & \leq \Big\|\int_{r(\psi)}^0 \phi'(s)\d s - \int_{r(\psi)}^0 \psi'(s)\d s\Big\|_{\R^n} + \Big\| \int_{r(\phi)}^{r(\psi)} \phi'(s)\d s\Big\|_{\R^n} \\ 
    & \leq \|1\|_{L_{2}(-h,0;\R^n)} \|\phi'-\psi'\|_{L_{2}(-h,0;\R^n)} + |r(\psi)-r(\phi)|\beta \\ & \leq \sqrt{h} \|\phi'-\psi'\|_{L_{2}(-h,0;\R^n)}+\beta \|r\|_{\textnormal{Lip}} \|\psi-\phi\|_{H^1(-h,0;\R^n)}.
    \end{align*}Putting the above inequalities together, we conclude the proof.
\end{proof}

%\begin{rem}
In applications, the delay functional $r$ assigning the time at which the unknown is to be evaluated on the right-hand side might not be Lipschitz continuous on the whole of $H^1(-h,0;\R^n)$. Moreover, $f$ as given in Theorem \ref{thm:sddcm} is only Lipschitz continuous on $V_\beta$ for all $\beta>0$. However, in all the cases we consider in Section \ref{sec:appl}, $r$ admits a Lipschitz continuous \emph{extension} to $H^1(-h,0;\R^n)$ and so does $f$ from Theorem \ref{thm:sddcm}. The reason for this is the Hilbert space structure of $H^1(-h,0;\R^n)$, which provides the possibility of defining projections onto closed convex subsets. Further to the comments in the introduction, we mention that in general also the metric projection\footnote{{Let $X$ be a Banach space, $C\subseteq X$. $C$ is called \textbf{proximinal}, if for all $x \in X$ there exists a \textbf{best-approximation of $x$ in $C$}, that is, an element $y\in C$ such that $\|x-y\|_X = \inf_{z\in C}\|x-z\|_X$. If $C$ is proximinal, the relation 
\[
   P_C \coloneqq \{(x,y)\in X\times X; y \text{ best-approximation of $x$ in $C$} \}
\] is called \textbf{metric projection}. If $X$ is a Hilbert space and $C$ closed and convex, $P_C$ is a mapping.}} in general Banach spaces is not as well-behaved as in the Hilbert space case. Indeed, any closed convex set admitting a best-approximation as in Proposition \ref{prop:normproj} requires the space considered to be reflexive and strictly convex excluding the space of continuous functions or continuously differentiable functions, see Theorem of Day--James in \cite[Chapter 5, p 436]{M98}. Moreover, if one was to extend $r$ assumed to be a $C^1$-function defined on a subset of $H^1(-h,0;\R^n)$ to a $C^1$-function on the whole space, one is confronted with a challenging matter on its own, see, e.g., \cite{AFK10} and the references therein.
%\end{rem}
For extending $r$ to a Lip\-schitz map on $H^1(-h,0;\R^n)$, we briefly recall the following abstract result for (minimal) projections in Hilbert spaces. 
\begin{prop}[{{see, e.g., \cite[p 142]{BK15}}}] \label{prop:normproj} Let $H$ be a Hilbert space, $C\subseteq H$ closed and convex. Then for all $x\in H$ there exists a unique $P_Cx\in C$ such that
\[
    \|x-P_Cx\|=\inf_{y\in C}\|x-y\|.
\]
Moreover, $P_C$ is a Lipschitz continuous projection with Lipschitz constant bounded by $1$.
\end{prop}
As a consequence of the latter result, we can now extend any Lipschitz continuous function on the full Hilbert space:
\begin{cor}[Lipschitz continuous extension]\label{cor:Lipf} Let $H$ be a Hilbert space, $\emptyset \neq C\subseteq H$ closed and convex, $Y$ a metric space. If $r\colon C\to Y$ is Lipschitz continuous, then $r_H\coloneqq r\circ P_C$ extends $r$ on $H$ with
\[
   \|r_H\|_{\textnormal{Lip}}=\|r\|_{\textnormal{Lip}}
\]
\end{cor}
\begin{proof}
Since by Proposition \ref{prop:normproj}, $P_Cx=x$ for all $x\in C$, it follows that 
\[
   r_H|_C=r,
\]which leads to $ \|r_H\|_{\textnormal{Lip}}\geq \|r\|_{\textnormal{Lip}}$. Proposition \ref{prop:normproj} yields that $r_H$ is Lipschitz continuous and $ \|r_H\|_{\textnormal{Lip}}\leq\|r\|_{\textnormal{Lip}}$.
\end{proof}

\section{Solution theory for state-dependent delay equations}\label{sec:sth}

Throughout this section, we let  $h,T>0$. We emphasise that $T<\infty$. We aim to provide a solution theory for any Lipschitz continuous given pre-history $\phi: [-h,0] \to \R^n$, of the following initial value problem on $(0,T)$,
\begin{equation}\label{eq:sdd}
   \begin{cases}
   x'(t) = g(t,x(t),x(t+r(x_{(t)}))),& t \in (0,T);\\
   x_{0} = \phi.
   \end{cases}
\end{equation}

The main assumptions for the equation just introduced are the following.
\begin{enumerate}
   \item[(H1)] $g\colon [0,T]\times \R^n\times \R^n\to \R^n$ is continuous and uniformly Lipschitz continuous with respect to the spatial variables, i.e., there exists $L\geq 0$ such that for all $t\in [0,T]$ and $x,y,u,v\in \R^n$ we have
   \[
      \|g(t,x,u)-g(t,y,v)\|\leq L(\|x-y\|+\|u-v\|).
   \]
   \item[(H2)] $r\colon H^1(-h,0;\R^n) \to [-h,0]$ is Lipschitz continuous.
\end{enumerate}
{We call $x\colon [-h,T]\to \R^n$ a \textbf{solution of \eqref{eq:sdd}}, if $x$ is continuous and weakly differentiable on $(0,T)$ and so that the first equation in \eqref{eq:sdd} holds for almost all $t\in (0,T)$ and $x|_{[-h,0]}=\phi$.}

We aim to show the following theorem. {In the following, for an interval $(a,b)\subseteq \R$ we say that $\psi\in L_2(a,b;\R^n)$ is \textbf{bounded}, if it is essentially bounded, that is, $\psi\in L_\infty(a,b;\R^n)$. In consequence, $\phi \in H^1(a,b;\R^n)$ has a bounded derivative if and only if $\phi' \in L_\infty(a,b;\R^n)$.}

\begin{thm}[Picard--Lindel\"of for SDDEs]\label{thm:wp} Let $g$ and $r$ be as in \textrm{(H1)} and \textrm{(H2)}. Then for all $\phi \in H^1(-h,0;\R^n)$ with bounded derivative, there exists a unique {solution} $x \in H^1(-h,T;\R^n)$ of \eqref{eq:sdd}.
\end{thm}

\begin{rem}\label{rem:sm}

(a) By the Sobolev embedding theorem, any { $x \in H^1(-h,T;\R^n)$} is continuous. This, however, does not require the derivative of $x$ to be continuous---even if $x$ is solution of \eqref{eq:sdd}. This has the following effect: the continuity of $t\mapsto \tau_t x$ as an $L_2$ (or $H^1$)-valued function (see Proposition \ref{prop:pre-h-cont}) forces the right-hand side of \eqref{eq:sdd} to be continuous in $t$. Thus, as $t\to0+$, the right-hand side converges to $g(0,\phi(0),\phi(r(\phi))$; the left-hand side does also converge; the limit, however, is $x'(0+)$ and is \emph{a priori unrelated} to $\phi'$. Therefore,  $\phi$ does not `see' the solution manifold. As a consequence, in Theorem \ref{thm:wp} there is no restriction on the initial value other than its (bounded) weak differentiability. In particular, the solution manifold is not needed for the formulation of the theorem.

(b) It follows from the equation, that any solution $x\in H^1(-h,T;\R^n)$ of \eqref{eq:sdd} satisfies $x\in H^1(-h,T;\R^n)\subseteq C([-h,T],
\R^n)$ and $x|_{[0,T]}\in C^1([0,T];\R^n)$. Hence, if $T>h$ and $x \in H^1(-h,T;\R^n)$ solves \eqref{eq:sdd} for some $\phi\in H^1(-h,0;\R^n)$ with bounded derivative; then $x(h+\cdot)$ belongs to the solution manifold as introduced in \cite{Walther}.

(c) A special case of Theorem \ref{thm:wp} is the classical form of the theorem of Picard--Lindel\"of, see e.g.~\cite[Theorem 4.2.6 (and proof)]{STW22}.

(d) Let $\phi\colon (a,b)\to \R^n$ for some $-\infty<a<b<\infty$. Then $\phi\in H^1(a,b;\R^n)$ with bounded derivative if and only if $\phi$ is Lipschitz continuous. Indeed, the `only if'-part is a straightforward consequence of integration by parts and a standard estimate for the integral. For the `if'-part note that Lipschitz continuity of $\phi$ implies boundedness of $\phi$ on $(a,b)$. Hence, $\phi\in L_2(a,b;\R^n)$. By Rademacher's theorem (see, e.g., \cite[Section 3.1.2]{EG15}), the differential quotient of $\phi$ exists at almost every point; by Lipschitz continuity, the differential quotient is bounded by $\|\phi\|_{\textnormal{Lip}}$. Hence, Lebesgue's dominated convergence theorem  implies that $\phi$ is weakly differentiable with derivative given by the almost everywhere limit of the difference quotients. As these are uniformly bounded, the weak derivative $\phi'$ is bounded and hence in $L_2(a,b;\R^n)$, eventually yielding $\phi\in H^1(a,b;\R^n)$ with bounded derivative.

(e) A combination of the arguments in (a) and (d) shows that any solution of \eqref{eq:sdd} is Lipschitz continuous. Indeed, for this it suffices to observe that $[0,T]\ni t\mapsto g(t,x(t),x(t+r(x_{t})))$ is continuous and that the pre-history $\phi$ is Lipschitz continuous by assumption.
\end{rem}

The next (non-)example of Theorem \ref{thm:wp} is a variant of the one provided in \cite{W70}. This example shows that the assumptions in Theorem \ref{thm:wp} are sharp. We emphasise, however, that the considered equation does satisfy (H1) and (H2) and, thus, Lipschitz continuous pre-histories imply  unique existence of solutions by Theorem \ref{thm:wp}.
\begin{example}\label{exa:sharp} Consider the state-dependent delay equation given by
\[
    x'(t) = -x(t- \min\{|x(t)|,2\})
\]
with $h=2$ and $T>0$. Then $g(t,x,u)=-u$ and $r\colon H^1(-h,0)\to [-h,0]$ given by $r(\phi)=-\min\{|\phi(0)|,2\}$ satisfy the assumptions of Theorem \ref{thm:wp}. Indeed, (H1) is clear. For (H2), we estimate for $\phi,\psi\in H^1(-h,0)$
\[
   |r(\phi)-r(\psi)| = |\min\{|\phi(0)|,2\}-\min\{|\psi(0)|,2\}|\leq |\phi(0)-\psi(0)|\leq\tfrac{3\sqrt{2}}{2} \|\phi-\psi\|_{H^1(-h,0)},
\]
by Theorem \ref{thm:set}. For $T>0$ sufficiently small both $x_1(t)=1+t$ and $x_2(t)=1+t-t^3$ satisfy the differential equation with the pre-history
\[
   \phi(t) \coloneqq \begin{cases} -1,& -2\leq t<-1\\
   3(t+1)^{2/3}-1,& -1\leq t\leq -\tfrac{\sqrt{27}-1}{\sqrt{27}},\\
   \tfrac{\sqrt{27}}{\sqrt{27}-1} t +1,& -\tfrac{\sqrt{27}-1}{\sqrt{27}}< t\leq 0,\\
   \end{cases}
\]
As $\phi\in H^1(-h,0)$ and uniqueness fails for this pre-history, a solution theory on all of $H^1(-h,0)$ cannot be expected even if $r$ is supposed to be globally Lipschitz a priori. Note that Theorem \ref{thm:wp} does not apply as $\phi$ is not Lipschitz continuous for its derivative blows up at $t=-1$.
\end{example}

The proof of Theorem \ref{thm:wp} requires several preliminary results. To begin with, as in the classical setting of ordinary differential equations, we shall reformulate the initial value problem \eqref{eq:sdd} into an integral equation. Similar to \cite{HT97}, we define for $\phi\in H^1(-h,0;\R^n)\subseteq C([-h,0];\R^n)$
\[
\hat{\phi} (t)\coloneqq \begin{cases} \phi(t),& t\in (-h,0]\\
\phi(0),& t\in (0,T)
\end{cases}
\]and aim to derive a fixed point problem for functions vanishing on $[-h,0]$. Note that $\hat\phi \in H^1(-h,T;\R^n)$. Indeed, {let $\psi\in C_c^\infty(-h,T)$ then
\begin{align*}
  -\int_{-h}^T \hat{\phi}(t)\psi'(t)\d  t& =  -\int_{-h}^0 \hat{\phi}(t)\psi'(t)\d  t -  \int_{0}^T \hat{\phi}(t)\psi'(t)\d  t \\
  &=  -\int_{-h}^0 {\phi}(t)\psi'(t)\d  t -  \int_{0}^T \phi(0)\psi'(t)\d  t \\
  &= {\phi}(0-)\psi(0-)-{\phi}(-h+)\psi(-h+)+ \int_{-h}^0 {\phi}'(t)\psi(t)\d  t - \phi(0+)\psi(0+) \\
  &=  \int_{-h}^0 {\phi}'(t)\psi(t)\d  t,
\end{align*}where we used the well-known integration by parts formula for $H^1$-functions, that ${\phi}(0-)={\phi}(0+)$ by the Theorem \ref{thm:set} and that $\psi$ vanishes at $-h$ and $T$.} The integral equation formulation now reads as follows.
\begin{lem}\label{lem:inteq} Let $x\in H^1(-h,T;\R^n)$ and $p\colon H^1(-h,0;\R^n)\to H^1(-h,0;\R^n)$ be a continuous mapping, $\phi\in H^1(-h,0;\R^n)$. Then the following conditions are equivalent:
\begin{enumerate}
  \item[(i)] $x$ satisfies 
  \begin{equation}\label{eq:xpint}
       x' = g(\cdot,x(\cdot), \ev(\id\times r)(p (x_{(\cdot)}))))\text{ as {equality in} }{L_2(0,T;\R^n)}
  \end{equation}
  and $x_0=\phi$.
  \item[(ii)] $x=y+\hat{\phi}$, where $y(t)=0$ on $(-h,0]$ and for $t\in (0,T)$
  \[
     y(t) = \int_0^t g(s,y(s)+\phi(0), \ev(\id\times r)(p((y+\hat{\phi})_{s}))) \d s
  \] and $y\in H^1(-h,T;\R^n)$.
\end{enumerate}
\end{lem}
\begin{proof}We note that the right-hand side of the differential equation (in (i) and the integrand in (ii)) is continuous by composition of continuous maps and the continuity of the shift taking values in $H^1$, see Proposition \ref{prop:pre-h-cont}.

(i)$\Rightarrow$(ii)  Integration over $(0,t)$ for some $t<T$ of both left- and right-hand sides of \eqref{eq:xpint} yields, using integration by parts on the left-hand side,
\[
   x(t)-x(0) = \int_{0}^t g(s,x(s),x(s+r(p (x_{(s)}))))\d s.
\]Defining $y(t)\coloneqq x(t)-\hat{\phi}(t)$, we obtain the desired integral equation; where we also used that $x(0)=\phi(0)$ due to continuity of both $x$ and $\phi$.

(ii)$\Rightarrow$(i) The claim follows using the fundamental theorem of calculus from the integral equation after having differentiated with respect to $t$. 
\end{proof}

\begin{rem} We apply Lemma \ref{lem:inteq} in two different cases, one for $p=\id$ and two for $p=\pi_\beta$, where
for $\beta>0$, we denote $\pi_\beta\coloneqq P_{V_\beta} \colon H^1(-h,0)\to H^1(-h,0)$, the orthogonal projection onto $V_\beta$ as defined in Proposition \ref{prop:normproj}. 
\end{rem}

Next, we prove exponential bounds for any solution of \eqref{eq:sdd}.

\begin{lem}\label{lem:bdd} Let $\phi\in H^{1}(-h,0;\R^n)$ with bounded derivative and let $x\in H^1(-h,T;\R^n)$ be a solution of \eqref{eq:sdd}. Then, for all $t\in [0,T]$, we have
\[
   \|x(t)\|\leq (\|\phi\|_{\infty}+\int_0^t \|g(s,0,0)\|\d s)\e^{2Lt}.
\]
\end{lem}
\begin{proof}
We start off by recalling that $x\in H^1(-h,T;\R^n)$ is continuous by Theorem \ref{thm:set}. For $t \in [0,T]$, we introduce
\[
    \kappa (t)\coloneqq \sup_{s\in [-h,t]}\|x(s)\|=\max_{s\in [-h,t]}\|x(s)\|.
\]
Since $x$ is continuous, so is $\kappa$. Hence, for $t\in [0,T]$, we estimate using Lemma \ref{lem:inteq},
\begin{align*}
   \|x(t) \|& \leq \|\phi(0)\| + \big\|\int_0^t g(s,x(s), x(s+r(x_{(s)})))\d s\big\| \\
   & \leq \|\phi(0)\| + \int_0^t \|g(s,x(s), x(s+r(x_{(s)})))-g(s,0,0)\|\d s  +\int_0^t \|g(s,0,0)\|\d s    \\
      & \leq \|\phi(0)\| + \int_0^t L(\|x(s)\|+\|x(s+r(x_{(s)}))\|)\d s  +\int_0^t \|g(s,0,0)\|\d s    \\
      & \leq \|\phi(0)\|+\int_0^t \|g(s,0,0)\|\d s + \int_0^t 2L\max_{\tau\in [-h,s]}\|x(\tau)\|\d s.
\end{align*}
Then, with $\alpha(t)\coloneqq \|\phi\|_{\infty}+\int_0^t \|g(s,0,0)\|\d s$, {where $\|\phi\|_\infty = \max_{s\in [-h,0]}\|\phi(s)\|$,} 
\begin{equation}\label{eq:kats}
\kappa(t)\leq \alpha(t)+2L\int_0^t \kappa(s)\d s.
\end{equation}
{Indeed, let $t\in [0,T]$ and $t^*\in [-h,t]$ be such that $\kappa(t) = \|x(t^*)\|$. If $t^*\in [-h,0]$, then $x(t^*)=\phi(t^*)$ and
\[
   \kappa(t)=\|x(t^*)\|=\|\phi(t^*)\|\leq \|\phi\|_\infty \leq \alpha(t)+2L\int_0^t \kappa(s)\d s.
\]
If $t^*\in (0,t]$, then by our estimate above, we deduce with $t^*\leq t$
\[
  \kappa(t)=\|x(t^*)\|_{\R^n} \leq \alpha(t^*)+2L\int_0^{t^*}\kappa(s)\d s\leq \alpha(t)+2L\int_0^{t}\kappa(s)\d s.
\]}
Thus, \eqref{eq:kats} and Gronwall's inequality imply
\[
   \|x(t) \|\leq \kappa(t)\leq \alpha(t)\e^{2Lt}.\qedhere
\]
\end{proof}
The following auxiliary statement is a variant of \cite[p 44 (see also Lemma 3.2.1)]{STW22} (see also \cite[p 1772]{Picard2009}) on finite time intervals.
\begin{prop}\label{prop:td} For all $\rho>0$, the operator
\[
I_\rho \colon    L_{2,\rho}(0,T) \to    L_{2,\rho}(0,T)
\]
given by 
\[
   I_\rho f(t)\coloneqq \int_0^t f(s)\d s 
\]
is continuous with $\|I_{\rho}\|\leq 1/\rho$.
\end{prop}
\begin{proof} {As $L_{2,\rho}(0,T)$ equals $L_{2}(0,T)$ with an equivalent norm, and since $L_2(0,T)\subseteq L_1(0,T)$ continuously, every $f\in L_{2,\rho}(0,T)$ is integrable on $(0,T)$. In particular, $\int_0^t f(s)\d s$ exists and, as a Lebesgue integral, it is independent of the representative.}  Let $f\in L_{2,\rho}(0,T)$. Then we compute
\begin{align*}
  \|I_\rho f\|^2 &= \int_0^T \big\|\int_0^t f(s)\d s\big\|^2 \e^{-2\rho t}\d t \\
  & \leq \int_0^T \big(\int_0^T \|f(s)\|\1_{(0,T)}(t-s) \d s\big)^2 \e^{-2\rho t}\d t \\
  &  = \int_0^T \big(\int_0^T \|f(s)\|\e^{-\rho s}\1_{(0,T)}(t-s)\e^{-\rho (t-s)} \d s\big)^2 \d t \\
  &  \leq \int_0^T \int_0^T \|f(s)\|^2\e^{-2\rho s}\1_{(0,T)}(t-s)\e^{-\rho (t-s)} \d s\big(\int_0^T \1_{(0,T)}(t-s)\e^{-\rho (t-s)} \d s\big) \d t
  \\
  & \leq \frac{1}{\rho} \int_0^T \int_0^T \|f(s)\|^2\e^{-2\rho s}\1_{(0,T)}(t-s)\e^{-\rho (t-s)} \d t\d s
   \leq \frac{1}{\rho^2}\|f\|_{L_{2,\rho}}^2.  \qedhere 
  \end{align*}
\end{proof}
We introduce, for $\rho\geq 0$,
\begin{equation}\label{eq:h100}
   H_{0,\rho}^1(0,T;\R^n)\coloneqq \{y\in H_\rho^1(-h,T;\R^n); y = 0 \text{ on } (-h,0)\}.\end{equation}
   As a consequence of the previous proposition, we get the following estimate.
   \begin{lem}\label{lem:intbdd}
     Let $\rho>0$, $f\in H_{0,\rho}^1(0,T;\R^n)$. Then $\|f\|_{L_{2,\rho}}\leq \frac{1}{\rho}\|f'\|_{L_{2,\rho}}$.
   \end{lem}
   \begin{proof}
   {By the integration by parts formula for $H^1$-functions and due to $f$ vanishing at $0$, we obtain $I_\rho (f|_{(0,T)}') = f|_{(0,T)}$. As $f=f'=0$ on $(-h,0)$, the claim follows from Proposition \ref{prop:td}.}
   \end{proof}
  We recall that $\pi_\beta\coloneqq P_{V_\beta} \colon H^1(-h,0;\R^n)\to H^1(-h,0;\R^n)$ denotes the projection onto $V_\beta$ as defined in Proposition \ref{prop:normproj}, where we also showed that $\|\pi_\beta\|_{\textnormal{Lip}}\leq 1$ as a mapping from $H^1(-h,0;\R^n)$ to $H^1(-h,0;\R^n)$.
\begin{thm}\label{thm:wpsdd1} Let $\beta>0$ and $\phi\in H^1(-h,0;\R^n)$. Then there is a unique $y\in H^1(-h,T;\R^n)$ satisfying
  \begin{equation}\label{eq:yt}
     y(t) = \int_0^t g(s,y(s)+\phi(0), \ev(\id\times r)(\pi_\beta( (y+\hat{\phi})_{s}))) \d s\quad(0<t<T)
  \end{equation} 
  and $y=0$ on $(-h,0)$.
\end{thm}
\begin{proof} {For $\rho>0$} consider the mapping $F\colon H_{0,\rho}^1(0,T;\R^n)\to H_{0,\rho}^1(0,T;\R^n)$ given by
  \[
     (Fy)(t) \coloneqq \int_0^t g(s,y(s), \ev(\id\times r)(\pi_\beta (y+\hat{\phi})_{(s)})) \d s\quad(0< t <T),
  \]
  where the integrand is continuous by Remark \ref{rem:conthatp} (continuity of the pre-history map as $H_\rho^1$-valued function), Proposition \ref{prop:normproj} (continuity of $\pi_\beta$), the continuity of $H^1$-functions by Theorem \ref{thm:set}  and the continuity of $g$.
  
  We shall prove that $F$ is a strict contraction for $\rho$ sufficiently large. For this, recall that $\|\pi_\beta\|_{\textnormal{Lip}}\leq 1$ by {the line just before Theorem \ref{thm:wpsdd1}}. For $u,v\in H_{0,\rho}^1(0,T)$, {using Proposition \ref{prop:td} (in the second line), Theorem \ref{thm:sddcm} (in the fourth line), and  Theorem \ref{thm:prehnorm} (and subsequent remark; in the seventh line),} we compute 
  \begin{align*}
    &\|Fu-Fv\|_{L_{2,\rho}}\\ &\leq \frac{1}{\rho}\big\| g(\cdot,u(\cdot), \ev(\id\times r)(\pi_\beta (u+\hat{\phi})_{(\cdot)})) -  g(\cdot,v(\cdot), \ev(\id\times r)(\pi_\beta (v+\hat{\phi})_{(\cdot)}))\big\|_{L_{2,\rho}(0,T;\R^n)} \\
    & \leq \frac{1}{\rho} L( \|u-v\|_{L_{2,\rho}(0,T;\R^n)}+\|\ev(\id\times r)(\pi_\beta (u+\hat{\phi})_{(\cdot)})-\ev(\id\times r)(\pi_\beta (v+\hat{\phi})_{(\cdot)})\|_{L_{2,\rho}(0,T;\R^n)})
    \\ & \leq \frac{1}{\rho} L( \|u-v\|_{L_{2,\rho}(0,T;\R^n)}\\
&   \quad + (2h^{1/2}+\beta \|r\|_{\textnormal{Lip}} + h^{-1/2})\| \|\pi_\beta((u+\hat{\phi})_{(\cdot)})-\pi_\beta (v+\hat{\phi})_{(\cdot)})\|_{H^1(-h,0;\R^n)}\|_{L_{2,\rho}(0,T)})
    \\ & \leq \frac{1}{\rho} L( \|u-v\|_{L_{2,\rho}(0,T;\R^n)}+(2h^{1/2}+\beta \|r\|_{\textnormal{Lip}} + h^{-1/2})\| \|u_{(\cdot)}-v_{(\cdot)}\|_{H^1(-h,0;\R^n)}\|_{L_{2,\rho}(0,T)})
    \\ & = \frac{1}{\rho} L( \|u-v\|_{L_{2,\rho}(0,T;\R^n)}+(2h^{1/2}+\beta \|r\|_{\textnormal{Lip}} + h^{-1/2})\| \hat{\Theta}u-\hat{\Theta}v\|_{L_{2,\rho}(0,T;H^1(-h,0;\R^n))}) 
    \\ & \leq \frac{1}{\rho} L( \|u-v\|_{L_{2,\rho}(0,T;\R^n)}+\frac{1}{\sqrt{2\rho}}(2h^{1/2}+\beta \|r\|_{\textnormal{Lip}} + h^{-1/2})\| u-v\|_{H^1_{\rho}(-h,T;\R^n)})\\ & \leq \frac{1}{\rho} L( \|u-v\|_{L_{2,\rho}(0,T;\R^n)}+\frac{1}{\sqrt{2\rho}}(2h^{1/2}+\beta \|r\|_{\textnormal{Lip}} + h^{-1/2})\| u-v\|_{H^1_{\rho}(0,T;\R^n)}),
  \end{align*}
  where the last equality follows from $u,v\in H_{0,\rho}^1(0,T)$.
  Similarly, we get
  \begin{align*}
&     \|(Fu)'-(Fv)'\|_{L_{2,\rho}(0,T;\R^n)}\\ & \leq L\|u-v\|_{L_{2,\rho}(0,T;\R^n)}+L\frac{1}{\sqrt{2\rho}}(2h^{1/2}+\beta \|r\|_{\textnormal{Lip}} + h^{-1/2})\|u-v\|_{H^1_{\rho}(0,T;\R^n)} \\
   &  \leq L\big(\frac{1}{\rho}+\frac{1}{\sqrt{2\rho}}(2h^{1/2}+\beta \|r\|_{\textnormal{Lip}} + h^{-1/2})\big)\|u-v\|_{H^1_{\rho}(0,T;\R^n)}  \end{align*}
  {We now choose $\rho$ large enough, in order to be able to identify $F$ as a strict contraction. Thus, by the contraction mapping principle, a unique existence of a fixed point in $y^*\in H_{0,\rho}^1(0,T;\R^n)$ is guaranteed. In order to prove the claim on uniqueness, let $y\in H^1(-h,T;\R^n)$ {satisfy \eqref{eq:yt}. Moreover, assume that $y=0$ on $(-h,0)$. } Hence,  $y\in H_{0,0}^1(0,T;\R^n)$. Since the sets $H_{0,0}^1(0,T;\R^n)$ and $H_{0,\rho}^1(0,T;\R^n)$ coincide as their respective norms are equivalent; we obtain $y\in H_{0,\rho}^1(0,T;\R^n)$ and, by uniqueness in $H_{0,\rho}^1(0,T;\R^n)$, it follows that $y=y^*$.  }
\end{proof}

Next, we provide a criterion for when we can drop the projection $\pi_\beta$ onto $V_\beta$.

\begin{thm}\label{thm:localsolth} Let $\phi\in H^1(-h,0;\R^n)$, $\beta>0$ and {choose $y\in H_{0,0}^1(0,T;\R^n) $ according to Theorem \ref{thm:wpsdd1}}. If $\phi\in V_\beta$ and
\[
   \|g(0,\phi(0),\phi(r(\phi)))\|<\beta,
\]
then for some $0<T_\beta \leq T$,  { $y +\hat{\phi}$ is a solution of \eqref{eq:sdd} on $(-h,T_\beta)$, that is, $y+\hat{\phi}$ is weakly differentiable on $(-h,T_\beta)$, $(y+\hat{\phi})_0=\phi$ and 
\[
   (y+\hat{\phi})'(t) = g(t,(y+\hat{\phi})(t),(y+\hat{\phi})(t+r((y+\hat{\phi})_{(t)}))),
\]for almost every $t\in (0,T_\beta)$. Moreover, $T_\beta$ may be chosen so that either $T_\beta=T$ or for all $0< t<T_\beta$, $\|y'(t)\|<\beta$ and $\|y'(T_\beta)\|=\beta$}
\end{thm}
\begin{proof} {Recall that by Theorem \ref{thm:set}, we may further argue having chosen the continuous representative of $y$. Then} $y(0)=0$. Moreover,  
\begin{align*}
    & \|(y+\hat{\phi})_s - \phi\|_{H^1(-h,0;\R^n)}^2 = \|\tau_s(y+\hat{\phi}) - \phi\|_{H^1(-h,0;\R^n)}^2\to 0
\end{align*}
as $s\to 0$, by Proposition \ref{prop:pre-h-cont} (and Remark \ref{rem:conthatp}). Hence, 
\[
    y'(t) = g(t,y(t)+\phi(0),\ev(\id\times r)(\pi_\beta (y+\phi)_{(t)}))\to g(0,\phi(0),\phi(r(\phi)))
\]
as $t\to 0+$. As $\|g(0,\phi(0),\phi(r(\phi)))\|<\beta$, we find $0<T_\beta\leq T$ such that 
\[
   \sup_{t\in (0,T_\beta)}\|y'(t)\|<\beta
\]
Hence, for all $t\in (0,T_\beta)$: $\pi_\beta( (y+\phi)_{(t)})= (y+\phi)_{(t)}$. { By Theorem \ref{thm:wpsdd1} and Lemma \ref{lem:inteq}, the part concerning existence of some $0<T_\beta\leq T$ is proven. For the remaining statement of the theorem, consider the mapping
\[
  \kappa \colon [0,T] \ni t\mapsto \|g(t,y(t)+\phi(0),\ev(\id\times r)(\pi_\beta (y+\phi)_{(t)}))\|.
\] By Proposition \ref{prop:pre-h-cont} and Remark \ref{rem:conthatp}, $\kappa$ is continuous. By assumption $\kappa(0)<\beta$. If $\kappa(t) < \beta$ for all $t\in [0,T)$; by the argument above, $\pi_\beta( (y+\phi)_{(t)})= (y+\phi)_{(t)}  $ and  $T=T_\beta$. If $\kappa(t) =\beta$ for some $t\in [0,T)$, we find $t_m \in [0,T)$ minimal such that $\kappa(t_m)=\beta$. In this case, using the differential equation satisfied by $y$, we infer $\|y'(t_m)\|=\beta$ and $\|y'(t)\|<\beta$ for all $t\in [0,t_m)$. Thus, $T_\beta=t_m$ has the properties of $T_\beta$ in case $T_\beta<T$.}
\end{proof}

Before we turn to a proof of Theorem \ref{thm:wp}, we establish uniqueness. 

\begin{lem}\label{lem:uniqueness} Let $\phi\in H^1(-h,0;\R^n)$ with bounded derivative. Then there exists at most one solution $x\in H^1(-h,T;\R^n)$ of \eqref{eq:sdd}.
\end{lem}
\begin{proof}
 Let $x^{(1)},x^{(2)}\in H^{1}(-h,T;\R^n)$ satisfy \eqref{eq:sdd}. By Remark \ref{rem:sm} (e), both $x^{(1)},x^{(2)}$ are Lipschitz continuous. Let $\beta>0$ be a common Lipschitz constant. {Then, for $j\in \{1,2\}$,
 \[
  (x^{(j)})' = g(\cdot,x^{(j)}(\cdot), \ev(\id\times r)((x^{(j)}_{(\cdot)}))))=g(\cdot,x^{(j)}(\cdot), \ev(\id\times r)(\pi_\beta(x^{(j)}_{(\cdot)}))))\text{ in }L_2(0,T;\R^n)
\]
  and $x^{(j)}_0=\phi$. By Lemma \ref{lem:inteq}, putting $y^{(j)} \coloneqq \hat{\phi} - x^{(j)}$ for $j\in \{1,2\}$, we get
 $y^{(j)}(t)=0$ on $(-h,0)$ and for $t\in (0,T)$
  \[
     y^{(j)}(t) = \int_0^t g(s,y^{(j)}(s)+\phi(0), \ev(\id\times r)(\pi_\beta ((y^{(j)}+\hat{\phi})_{(s)}))) \d s
  \] and $y^{(j)}\in H^1(-h,T;\R^n)$. By Theorem \ref{thm:wpsdd1} fixed points for this integral equation are unique in $ H^1(-h,T;\R^n)$ and so $y^{(1)}=y^{(2)}$ or, equivalently, $x^{(1)}=x^{(2)}$, as claimed.}
\end{proof}

The preparations are complete for the main result of the present paper.

\begin{proof}[Proof of Theorem \ref{thm:wp}] 
Uniqueness has been established in Lemma \ref{lem:uniqueness}. By assumption, we find $\beta_0>0$ such that $\phi\in V_{\beta_0}$ and $\|g(0,\phi(0),\phi(r(\phi)))\|\leq\beta_0$. Let $\beta>\beta_0$. Then, by Theorems \ref{thm:wpsdd1} and \ref{thm:localsolth}, there exists $0<T_\beta\leq T$ and a unique solution $x$ on $(0,T_\beta)$, where $T_\beta$ is as in the concluding statement in Theorem \ref{thm:localsolth}; that is, such that either $T_\beta=T$ or for all $t<T_\beta$, $\|x'(t)\|<\beta$ and $\|x'(T_\beta)\|=\beta$. 
{If $T=T_\beta$, we are done. If $T_\beta<T$, since $\phi\in V_{2\beta}$ by Theorems \ref{thm:wpsdd1} and \ref{thm:localsolth}, there exists $0<T_{2\beta}\leq T$ and a unique solution $\tilde{x}$ on $(0,T_{2\beta})$ with either $T=T_{2\beta}$ or $\|\tilde{x}'(T_{2\beta}\|=2\beta$. Since $\|x'(t)\|<\beta\leq 2\beta$ for all $t\in (0,T_\beta)$ and the solution being unique particularly on $(0,T_\beta)$, it follows that $\tilde{x}=x$ on $(0,T_\beta)$. We drop $\tilde{\phantom{x}}$ again and refer to the solution on $(0,T_{2\beta})$ by $x$ instead. Since $\|x'(T_{\beta})\|=\beta <\|x'(T_{2\beta})\|=2\beta$ and $\|x'(t)\|<\beta$ for all $t\in [0,T_\beta)$ it follows that $T_{2\beta}>T_\beta$.} We obtain a sequence $(T_{k\beta})_{k\in \N}$ of strictly positive numbers, which is increasing and either eventually constant $T$; or for all $k\in \N$, $k\geq 1$, we have $T_{k\beta}<T$. In the former case we are done. Next, we assume, by contradiction, that the latter case holds { and define $T_\infty\coloneqq \lim_{k\to\infty} T_{k\beta}\leq T$}. This implies that 
\begin{equation}\label{eq:gtinf}\|g(T_{k\beta},x(T_{k\beta}),x(T_{k\beta}+r(x_{(T_{k\beta})})))\|=\|x'(T_{k\beta})\|=k\beta \to \infty \quad (k\to\infty).
\end{equation}
It follows that $\sup_{t\in [-h,T_\infty]}\|x(t)\|=\infty$. {Indeed, if, by contradiction, $\sup_{t\in [-h,T_\infty)}\|x(t)\|<\infty$, then, as $g$ is bounded on bounded sets by continuity of $g$, we get a contradiction to \eqref{eq:gtinf}}. Lemma \ref{lem:bdd}, however, implies $\sup_{t\in [-h,T_\infty)}\|x(t)\|<\infty$, which is a contradiction. Hence, $(T_{k\beta})_{k\in \N}$ is eventually constant equal to $T$ and the claim follows.
\end{proof}

We conclude this section with a permanence principle: Given a solution $x$ of \eqref{eq:sdd}, for every $\varepsilon>0$ there exists a time $0<T_\varepsilon\leq T$ such that both the pre-history of $x$ and $x'$ do not deviate from $\phi$ and $g(0,\phi(0),\phi(r(\phi)))$ more than $\varepsilon$, respectively. 

\begin{thm}[Permanence principle]\label{thm:pp} Let $\phi\in H^1(-h,0;\R^n)$ with bounded derivative and let $x\in H^1(-h,T;\R^n)$ be a solution of \eqref{eq:sdd}. Let $\varepsilon>0$. Then there exists $T_\varepsilon>0$ such that for all $0\leq t\leq T_\varepsilon$
\[
   \sup_{s\in [-h,0]} (\|x(s+t)-\phi(s)\|_{\R^n})+\|x'(t)-g(0,\phi(0),\phi(r(\phi)))\|_{\R^n}\leq \varepsilon.
\]
\end{thm}
\begin{proof}
The solution $x\in H^1(-h,T;\R^n)$ agrees with $\phi$ on $[-h,0]$. Since $x$ is continuous by Theorem \ref{thm:set}, it is uniformly continuous and bounded on the compact set $[-h,T]$. {Since} $g$ is uniformly continuous on $[0,T]\times K\times K$ with $K=x[[-h,T]]$ we find $\delta_0>0$ such that for all $t,s\in [0,T], u,v,w,z\in K$ with $|t-s|+\|u-v\|+\|w-z\|\leq \delta_0$, we infer
\[
    |g(t,u,w)-g(s,v,z)|\leq \varepsilon/2.
\]
By uniform continuity of $x$, we find $\delta_1>0$ such that $\|x(s+t)-\phi(s)\|\leq \min\{\varepsilon/2,\delta_0/3\}$ for all $s\in [-h,0]$ as long as $0\leq t\leq \delta_1$. By continuity of the shift (see Proposition \ref{prop:pre-h-cont} and Remark \ref{rem:conthatp}) and continuity of $r$, we find $\delta_2>0$ such that $\|r(x(t+\cdot))-r(\phi)\|\leq \delta_1/2$ for all $0\leq t\leq\delta_2$. 

Thus, for $T_\varepsilon\coloneqq \min\{\delta_0/3,\delta_2,\delta_1/2\}$ and $0< t\leq T_\varepsilon$ we estimate
\[
    \|x(s+t)-\phi(s)\|\leq \varepsilon/2\quad(s\in [-h,0])
\]
and
\begin{align*}
   \|x'(t)-g(0,\phi(0),\phi(r(\phi)))\|&=\|g(t,x(t),x(t+r(x_{t})))-g(0,\phi(0),\phi(r(\phi)))\|\leq \varepsilon/2,
\end{align*}as 
\begin{multline*}
|t|+\|x(t)-\phi(0)\|+\|x(t+r(x_{t}))-\phi(r(\phi))\|\\ \leq \tfrac{\delta_0}{3}+\tfrac{\delta_0}{3}+\|x(t+r(x_{t})-r(\phi)+r(\phi))-\phi(r(\phi))\|\leq \delta_0,
\end{multline*}
where we used that $r(\phi)\in [-h,0]$ and $|t+r(x_{t})-r(\phi)|\leq\delta_1$.
\end{proof}

Before we come to some {concrete examples} we shall consider the extension of the techniques developed here to functional differential equations.

\section{Solution theory for functional differential equations}\label{sec:fde}

An adaptation of the techniques carried out in the proof of Theorem \ref{thm:wp} will show the following more general theorem for functional differential equations (FDEs). The right-hand side of the differential equation being more general as opposed to \eqref{eq:sdd}, we { cannot guarantee existence of solutions throughout the entire given time-horizon, $(-h,T)$, where} again, we let $h,T>0$ throughout this section. We say that $G\colon [0,T]\times H^1(-h,0;\R^n)\to \R^m$ is \textbf{almost uniformly Lipschitz continuous}, if  it is continuous and if for all $\beta>0$, there exists $L\geq 0$ such that for all $u,v\in V_\beta$ and $t\in [0,T]$ 
\[
   \|G(t,u)-G(t,v)\|_{\R^m}\leq L\|u-v\|_{H^1(-h,0;\R^n)}
\] 
Here, we shall {consider and prove} Theorem \ref{thm:fdewpint} from the introduction and present the following more detailed version of it.

\begin{thm}[Picard--Lindel\"of for FDEs]\label{thm:wpfde} Let $G\colon [0,T]\times H^1(-h,0;\R^n)\to \R^n$ be almost uniformly Lipschitz continuous. Then for all $\phi\in H^1(-h,0;\R^n)$ with bounded derivative, there exists $0<T_0\leq T$ and { a unique } $x\colon (-h,T_0)\to \R^n$ such that for all $T_1\in (0,T_0)$,  $x|_{(-h,T_1)}\in H^1(-h,T_1;\R^n)$ satisfies
\begin{align*}\tag{FDE}
    x'(t)=G(t,x_{t})\quad (t\in (0,T_1)),\quad x_0=\phi.
\end{align*}
If $T_0$ is chosen maximally, then either $\|x\|_{\textnormal{Lip}}=\infty$ or both $T_0=T$ and $x\in H^1(-h,T;\R^n)$. 
\end{thm}
\begin{proof}
{In order to show that there exists $0<T_0\leq T$ and $x$ with the mentioned properties, we let $\beta>0$.} Then
\[
  \Phi_\beta\colon H_{0,\rho}^1(0,T;\R^n)\to H_{0,\rho}^1(0,T;\R^n)
\]
defined by
\begin{equation}\label{eq:phib}
    (\Phi_\beta y)(t) = \int_0^{\max\{t,0\}} G(s, \pi_\beta (y+\hat{\phi})_{s})\d s,
\end{equation}
where $t\in (-h,T)$, is Lipschitz continuous and a strict contraction if $\rho>0$ is large enough, where again $\pi_\beta$ is the {Lipschitz} continuous projection onto $V_\beta$. {For this, note for $y\in H_{0,\rho}^1(0,T;\R^n)$ that $[0,T]\ni s\mapsto G(s, \pi_\beta (y+\hat{\phi})_{(s)})$ is continuous by Remark \ref{rem:conthatp}, the continuity of $\pi_\beta$ and continuity of $G$. This leads to the well-definedness of the integral. Moreover, by the fundamental theorem of calculus, $\Phi_\beta y$ is continuously differentiable on $[0,T]$ with $\Phi_\beta y(0)=0$ and $\Phi_\beta y =0$ on $(-h,0)$. {Hence, using the well-known integration by parts formula for $H^1$, it follows that $\Phi_\beta y\in H_{0,\rho}^1(0,T;\R^n)$: Indeed, let $\psi \in C_c^\infty(-h,T)$. We compute
\begin{align*}
  -\int_{-h}^T \psi'(s) \Phi_\beta y(s)\d s &=  -\int_{0}^T \psi'(s) \Phi_\beta y(s)\d s \\ & = -\psi'(0) \Phi_\beta y(0) + \int_{0}^T \psi(s) (\Phi_\beta y)'(s)\d s =  \int_{0}^T \psi(s) (\Phi_\beta y)'(s)\d s,
\end{align*}
which shows $\Phi_\beta y\in H_{0,\rho}^1(0,T;\R^n)$.} Hence, $\Phi_\beta$ is well-defined. Moreover, for $y,z\in H_{0,\rho}^1(0,T;\R^n)$ we compute using Lemma \ref{lem:intbdd}
\begin{align*}
  \|\Phi_\beta y-\Phi_\beta z\|_{L_{2,\rho}(-h,T)}^2 \leq  \frac{1}{\rho^2} \|(\Phi_\beta y)'-(\Phi_\beta z)'\|_{L_{2,\rho}(0,T)}^2.
\end{align*}
Moreover, we find $L\geq 0$ depending on $\beta$, such that
\begin{align*}
\|(\Phi_\beta y)'-(\Phi_\beta z)'\|_{L_{2,\rho}(0,T)}^2 & = \int_0^T \|G(s, \pi_\beta (y+\hat{\phi})_{s})-G(s, \pi_\beta (z+\hat{\phi})_{s})\|^2 \e^{-2\rho s} \d s  \\
&\leq  L^2 \int_0^T \|\pi_\beta (y+\hat{\phi})_{s}- \pi_\beta (z+\hat{\phi})_{s}\|_{H^1(-h,0;\R^n)}^2 \e^{-2\rho s}\d s \\
& \leq  L^2 \int_0^T \| (y+\hat{\phi})_{s}-  (z+\hat{\phi})_{s}\|_{H^1(-h,0;\R^n)}^2 \e^{-2\rho s}\d s \\
 & =  L^2 \int_0^T \| y_{s}-  z_{s}\|_{H^1(-h,0;\R^n)}^2 \e^{-2\rho s}\d s \\
  & =  \frac{1}{2\rho} L^2 \| y-  z\|_{H_{0,\rho}^1(-h,T;\R^n)}^2,
\end{align*}
where in the last step we used Remark \ref{rem:thetaH}. Thus, for $\rho>0$ large enough, $\Phi_\beta$ has a unique fixed point $y^{(\beta)} \in H_{0,\rho}^1(0,T;\R^n)=H_{0,0}^1(0,T;\R^n)\subseteq H^1(-h,T;\R^n)$.
Next, we let $\beta>\max\{\|\phi\|_{\textnormal{Lip}},G(0,\phi)\}$. Then, as argued as in Remark \ref{rem:thetaH}, we get
\[
   (( y^{(\beta)}+\hat{\phi})_s)' =    ( (y^{(\beta)})'+\hat{\phi}')_s\quad(s\in [0,T])
\]
and, since $y^{(\beta)}$ is the fixed point of $\Phi_\beta$,
\[
(y^{(\beta)})'(t)=G(t, \pi_\beta (y^{(\beta)}+\hat{\phi})_{t})\quad(t\in (0,T)).
\]
By $\|G(0,\pi_\beta (y+\hat{\phi})_{0})\| =\|G(0,\pi_\beta \phi)\|=\|G(0,\phi)\|<\beta$ and the continuity of $[0,T]\ni t\mapsto G(t, \pi_\beta (y+\hat{\phi})_{t})$ it follows that $(y^{(\beta)}+\hat{\phi})_s \in V_\beta$ for all $s\in [0,T_\beta]$ with $0<T_\beta\leq T$ given by
\begin{equation}\label{eq:tb}
   T_\beta \coloneqq \begin{cases} T, \quad \quad \text{if, for all $t\in [0,T)$, }\|G(t, \pi_\beta (y^{(\beta)}+\hat{\phi})_{t})\|<\beta;\\
   \min\{ t\in [0,T); \|G(t, \pi_\beta (y^{(\beta)}+\hat{\phi})_{t})\|=\beta\},  \quad \quad \text{ otherwise}.   
   \end{cases}
\end{equation} Thus,
\[
   (y^{(\beta)})'(t)=G(t, \pi_\beta (y^{(\beta)}+\hat{\phi})_{t})=G(t, (y^{(\beta)}+\hat{\phi})_{t})\quad (t\in (0,T_\beta))
\]
and the existence part is proven with $T_0=T_\beta$ and $x=(y^{(\beta)}+\hat{\phi})|_{(-h,T_\beta)}$. {In particular, for all $T_1<T_0$, $x|_{(-h,T_1)}$ satisfies (FDE) and the initial condtion.} \newline
For uniqueness, let {$0<T_0\leq T$ and $w\colon (-h,T_0)\to \R^n$ be such that $w|_{(-h,T_1)}\in  H^1(-h,T_1;\R^n)$ satisfies the (FDE) and the initial condition for all $0<T_1<T_0$. Let $T_1<T_0$ and define $x\coloneqq w|_{(-h,T_1)}$.} Then $x$ is Lipschitz continuous. Indeed, since $\phi$ is Lipschitz, $x$ is Lipschitz if we show $x$ has a bounded derivative on $(0,T_1)$, see Remark \ref{rem:sm}(d). By Remark \ref{rem:conthatp}, the mapping $[0,T_1]\ni t\mapsto x_{t}\in H^1(-h,0;\R^n)$ is continuous. Hence, the image, $X$, of said mapping is a compact subset of $H^1(-h,0;\R^n)$. Since $G$ is continuous, $G[[0,T_1]\times X]\subseteq \R^n$ is compact and hence bounded. Thus, $\|x'(t)\|\leq \max G[[0,T_1]\times X]<\infty$ for all $t\in [0,T_1]$. Moreover, note that $y =x-\hat{\phi} \in H_{0,0}^1(0,T_1;\R^n)$ (recall \eqref{eq:h100}) satisfies the integrated version of (FDE), that is,
\[
   (\Phi y)(t) = \int_0^{\max\{t,0\}} G(s,  (y+\hat{\phi})_{s})\d s\quad (t\in [0,T_1]).
\]Since $x$ (and, hence, $y$) is Lipschitz continuous, we find $\beta>0$ such that $x =y+\hat{\phi} \in V_{\beta/2 }$ and $y$ is the fixed point of $\Phi_{\beta}$ defined as in \eqref{eq:phib} for $t\in (-h,T_1)$. \newline
If, now, $\tilde{x}\in H^1(-h,T_1;\R^n)$ satisfies (FDE) and the initial conditions, by the same arguments as for $x$, we find $\tilde{y} = \tilde{x}-\hat{\phi} \in H_{0,0}^1(0,T_1;\R^n)$ being the fixed point of $\Phi_{\tilde{\beta}}$ for some $\tilde{\beta}>0$. Without loss of generality, $\beta\geq \tilde{\beta}$. Since 
\begin{align*}
  \tilde{y} & = \Phi_{\tilde\beta}  \tilde{y} \\
  & =\int_0^{\max\{\cdot,0\}} G(s, \pi_{\tilde\beta} (\tilde{y}+\hat{\phi})_{s})\d s\\
  & =\int_0^{\max\{\cdot,0\}} G(s,  (\tilde{y}+\hat{\phi})_{s})\d s \\
  & =\int_0^{\max\{\cdot,0\}} G(s, \pi_{\beta} (\tilde{y}+\hat{\phi})_{s})\d s  = \Phi_{\beta}\tilde{y} ,
\end{align*}
and by uniqueness of the fixed point of $\Phi_\beta$, which is implied by the contraction mapping principle, we get $\tilde{y}={y}$ and, hence, $\tilde{x}=x$. {Hence, any solution on $(-h,T_1)$ coincides with the restriction of $w$ to $(-h,T_1)$ and, thus, uniqueness follows.}\newline
Finally, recall the definition of $T_\beta$ in \eqref{eq:tb} with $x^{(\beta)} \coloneqq y^{(\beta)}+\hat{\phi} \in H^1(-h,T;\R^n)$ where $y^{(\beta)} \in H^1(-h,T;\R^n)$ is the fixed point of $\Phi_\beta$ given as in \eqref{eq:phib} for $t\in (-h,T)$. {By the definition of $T_\beta$ it follows that $(x^{(\beta)})_t \in V^\beta$ for all $0\leq t\leq T_\beta$.}  Consider the sequence $(T_{k\beta})_{k\in \N}$. Next we show that $(T_{k\beta})_{k\in \N}$ is increasing. 
Indeed, let $k_1,k_2\in \N$ with $1\leq k_1<k_2$. By contradiction, assume $T_{k_2\beta}< T_{k_1\beta}$. By the uniqueness part, $x^{(k_1\beta)}=x^{(k_2\beta)}$ on $(-h,T_{k_2\beta})$.  By construction, $ \|G(T_{k_2\beta}, \pi_{k_2\beta} (x^{(k_1\beta)})_{(T_{k_2\beta})})\|=k_2\beta$ with $T_{k_2\beta}$ minimal. The intermediate value theorem implies due to continuity of $t\mapsto \|G(t, \pi_{k_2\beta} (x^{(k_1\beta)})_{t})\|$ that we find a minimal $0<T^*<T_{k_2\beta}$ such that 
\[
\|G(T^*, \pi_{k_2\beta} (x^{(k_1\beta)})_{T^*})\|=k_1\beta.
\]{
Since $(x^{(k_1\beta)})_{t} \in V_{k_1\beta}$ for all $0\leq t<T_{k_1\beta}$ and $V_{k_1\beta}\subseteq V_{k_2\beta}$, we deduce 
\[\pi_{k_2\beta} (x^{(k_1\beta)})_{t}=\pi_{k_2\beta} \pi_{k_1\beta} (x^{(k_1\beta)})_{t}) =\pi_{k_1\beta} (x^{(k_1\beta)})_{t}). 
\]Hence, as $0<T^*<T_{k_2\beta}<T_{k_1\beta}$,
\[k_1\beta=\|G(T^*, \pi_{k_2\beta} (x^{(k_1\beta)})_{T^*})\|=\|G(T^*, \pi_{k_1\beta} (x^{(k_1\beta)})_{T^*})\|.
\]
Since $T^*<T_{k_1\beta}$ and $T_{k_1\beta}$ is minimal with $k_1\beta=\|G(T^*, \pi_{k_1\beta} (x^{(k_1\beta)})_{T^*})\|$ we obtain a contradiction.}\newline
Thus, $(T_{k\beta})_{k\in \N}$ is either eventually constant equal $T$, in which case $x^{(k\beta)}\in H^1(-h,T;\R^n)$ for some $k\in \N$ and $x=x^{(k\beta)}$ satisfies (FDE) on $(0,T)$ and the second alternative holds for $T_0=T$. Otherwise, $T_{k\beta}<T$ for all $k\in \N$. By the uniqueness part we may set $x\coloneqq x^{(\beta)}$ on $[0,T_\beta)$ and
\[
    x \coloneqq x^{(k+1)\beta}\text{ on }[T_{k\beta},T_{(k+1)\beta})\text{ for all }k\in \N, k\geq 1.
\]  Defining $T_0\coloneqq \lim_{k\to\infty}T_{k\beta}\leq T$, we obtain $x$ as claimed to exist in the theorem. Moreover, it follows that $\|x'(T_{k\beta})\|=\|G(T_{k\beta}, \pi_{k\beta} (y^{(k\beta)}+\hat{\phi})_{T_{k\beta}})\|=k\beta\to \infty$ as $k\to\infty$. This leads to $\|x\|_{\textnormal{Lip}}=\infty$.
}\end{proof}
{
We apply our findings to a state-dependent delay differential equation with multiple delays.
\begin{thm}\label{thm:muldel} Let $\phi\in H^1(-h,0;\R^n)$ with bounded derivative, $r_1,\ldots, r_m : H^1(-h,0;\R^n)\to [-h,0]$ Lipschitz continuous. Let $g\colon [0,T]\times (\R^n)^m\to \R^n$ be continuous. Assume there exists $L\geq 0$ such that for all $t\in [0,T]$, $u_1,\ldots, u_m, v_1,\ldots, v_m\in \R^n$: 
\[
   \| g(t,u_1,\ldots, u_m)-g(t,v_1,\ldots, v_m)\|_{\R^n}\leq L\sum_{j=1}^m \|u_j-v_j\|_{\R^n}.
\]
Then there exists a unique $x\in H^1(-h,T;\R^n)$ such that $x_0 =\phi$ and
\[
   x'(t) = g(t, x(t+r_1(x_t)), \ldots, x(t+r_m(x_t)))\quad (t\in (0,T)).
\]
\end{thm}
\begin{proof} We apply Theorem \ref{thm:wpfde} with 
\[
    G \colon [0,T]\times H^1(-h,0;\R^n)\to \R^n, (t,\phi) \mapsto g(t, \phi(r_1(\phi)), \ldots, \phi(r_m(\phi))).
\]
Then $G$ is continuous as composition of continuous mappings. Moreover, let $\beta>0$. Then, if $\phi,\psi \in V_\beta$, by Theorem \ref{thm:sddcm} applied to $g=\id$ on $\R^n$ and $r=r_j$ for some $j\in \{1,\ldots, m\}$, we get
\[
  \|\phi(r_j(\phi))-\psi(r_j(\psi))\|_{\R^n}\leq (2h^{1/2}+\beta \|r_j\|_{\textnormal{Lip}} + h^{-1/2})\|\phi-\psi\|_{H^1(-h,0;\R^n)}.
\]
Hence, for all $\phi,\psi\in V_\beta$ we get for all $t\in [0,T]$
\begin{align*}
 \|G(t,\phi)-G(t,\psi)\|_{\R^n} &\leq L\sum_{j=1}^m   \|\phi(r_j(\phi))-\psi(r_j(\psi))\|_{\R^n} \\ &\leq \big(L\sum_{j=1}^m (2h^{1/2}+\beta \|r_j\|_{\textnormal{Lip}} + h^{-1/2})\big)\|\phi-\psi\|_{H^1(-h,0;\R^n)}.
\end{align*}
Thus, $G$ is almost uniformly Lipschitz continuous. Note that for $x\in H^1(-h,T)$ and $t\in [0,T]$, we have
\[
   G(t,x_t) = g(t, x_t(r_1(x_t)), \ldots, x_t(r_m(x_t))) = g(t, x(t+r_1(x_t)), \ldots, x(t+r_m(x_t))).
\] By Theorem \ref{thm:wpfde} we find $T_0\leq T$ chosen maximally and a unique $x$ satisfying (FDE) for the particular $G$ considered here. In order to prove the final claim, $x\in H^1(-h,T;\R^n)$, we use the alternatives for $x$ defined on $(-h,T_0)$ with  $T_0$ maximal  in Theorem \ref{thm:wpfde}.  Thus, it suffices to show that $\|x\|_{\textnormal{Lip}}<\infty$. For this, we assume by contradiction, $\|x\|_{\textnormal{Lip}}=\infty$. Using the (FDE), we get 
\[
\lim_{t\to T_0-} \|g(t, x_t(r_1(x_t)), \ldots, x_t(r_m(x_t)))\|=\lim_{t\to T_0-} \|G(t,x_t)\|=\infty.
\] Since $g$ is continuous and thus bounded on bounded sets we deduce that 
\begin{equation}\label{eq:kainf}
\kappa(s)\coloneqq \max_{\tau\in [-h,s]}\|x(\tau)\|\to\infty
\end{equation} as $s\to T_0-$. Integrating the (FDE) on $[0,t]$ for some $t\in (0,T_0)$, we compute
\begin{align*}
 \|x(t)\| & = \|\phi(0) + \int_0^t G(s,x_s)\d s \| \\
           & \leq  \|\phi(0)\| + \int_0^t \|G(s,x_s)\|\d s =\|\phi(0)\| + \int_0^t \|g(s, x_s(r_1(x_s)), \ldots, x_s(r_m(x_s)))\|\d s  \\
           & \leq  \|\phi(0)\| + \int_0^t \|g(s,0, \ldots, 0)\|\d s  + L\sum_{j=1}^m \int_0^t \| x_s(r_j(x_s))\|\d s   \\
           &\leq c + L m \int_0^t \kappa(s)\d s,           
\end{align*} 
where $c\coloneqq \|\phi\|_\infty+ \int_0^T \|g(s,0, \ldots, 0)\|\d s<\infty$. It follows that
\[
   \kappa(t)\leq c + L m \int_0^t \kappa(s)\d s
\]
and Gronwall's inequality implies
\[
   \kappa(t)\leq c \e^{Lm t}.
\]
Hence, $\lim_{t\to T_0-}\kappa(t) <\infty$ contradicting \eqref{eq:kainf}. Hence, $x\in H^1(-h,T;\R^n)$.
\end{proof}
}
%%
%%The extension technique and local solution theory just developed can be applied to an even larger class of functions using Kirszbraun's theorem. The remarkable fact is that there is \emph{no} condition on $\dom(f)$ making this extension prossible.
%%
%%\begin{thm}[Kirszbraun's Theorem, see \cite{K34}] Let $H_1,H_2$ be Hilbert spaces, $f\colon \dom(f)\subseteq H_1\to H_2$ be Lipschitz continuous. Then there exists $F\colon H_1\to H_2$ with $F|_U=f$ such that $\|F\|_{\textnormal{Lip}}=\|f\|_{\textnormal{Lip}}$.
%%\end{thm}

\section{Applications}\label{sec:appl}

In this section, we consider two different state-dependent delay equations. One example is an elementary problem from Newtonian mechanics and the other one is an application from mathematical biology. In either cases, we shall not go through the model assumptions in full detail as this section is merely meant to illustrate the versatility of the solution theory developed in the previous section. More so, we rather focus on the delay functional $r$ in each application. We briefly sketch how the model assumptions can be relaxed to obtain a more general solution theory for the individual equation at hand.

\subsection{On a problem from mathematical biology}

In this section, we revisit the example from cell population biology, which has been analysed in \cite{GW14,BGR21}. Here, we will use the slightly more general theorem on the solution theory for a class of functional differential equations, see Theorem \ref{thm:wpfde}. Again, we emphasise that we merely require appropriate Lipschitz continuity; the assumptions are then similar to the ones developed in \cite{BGR21}. Note that, however, the main abstract assumption in \cite{BGR21} establishing both uniqueness and existence is that of so-called `almost locally Lipschitz' continuity for the right-hand side of the functional differential equation involved. More precisely, let $h>0$. Then $f\colon \dom(f)\subseteq C([-h,0];\R^n)\to \R^m$ is called \textbf{almost locally Lipschitz}, if $f$ is continuous and for all $\phi_0\in \dom(f)$ and $\beta>0$ there exists $\delta>0$ and $L\geq0$ such that
\[
   \|f(u)-f(v)\|_{\R^m}\leq L\|u-v\|_{C([-h,0];\R^n)},
\] whenever $u,v\in B_{C[-h,0]}[\phi_0,\delta]$, the closed $\delta$-ball around $\phi_0$, and $\|u\|_{\textnormal{Lip}},\|v\|_{\textnormal{Lip}}\leq \beta$. Any almost locally Lipschitz continuous mapping can be extended to an almost uniformly Lipschitz continuous mapping in the sense of Theorem \ref{thm:wpfde} as long as $\dom(f)\subseteq C([-h,0];\R^n)$ is open. In fact, in this case, one can even extend it to a Lipschitz continuous mapping as the following result shows. We consider the autonomous case here only, the non-autonomous situation as treated in \cite{HV93} can be addressed similarly.

\begin{prop}\label{prop:aLc} Let $f\colon \dom(f)\subseteq C([-h,0];\R^n)\to \R^m$ almost locally Lipschitz with $\dom(f)\subseteq C([-h,0];\R^n)$ open. Then for all Lipschitz continuous $\phi_0\in \dom(f)$ there exists $F\colon H^1(-h,0;\R^n)\to \R^m$ Lipschitz continuous and $\delta>0$ such that for all $\beta>0$
\[
   F|_{ B_{C[-h,0]}[\phi_0,\delta]\cap V_\beta}=f|_{B_{C[-h,0]}[\phi_0,\delta]\cap V_\beta}
\]
\end{prop}
\begin{proof}
At first, we show for $\beta>\|\phi_0\|_{\textnormal{Lip}}$, $D\coloneqq V_\beta\cap B_{C[-h,0]}[\phi_0,\delta]\subseteq H^1(-h,0;\R^n)$ is convex and closed and non-empty. Since $V_\beta\subseteq H^1(-h,0;\R^n)$, we get  $D\subseteq H^1(-h,0;\R^n)$. Convexity of $D$ being evident as an intersection of convex sets, it remains to show closedness. For this, let $(\psi_k)_{k\in \N}$ in $D$ converging in $H^1(-h,0;\R^n)$ to some $\psi\in H^1(-h,0;\R^n)$. By Theorem \ref{thm:set}, it follows that $\psi\in B_{C[-h,0]}[\phi_0,\delta]$. Remark \ref{rem:Vbeta} confirms that $\psi\in V_\beta$.

For the construction of the extension, let $\phi_0\in \dom(f)$ be Lipschitz continuous. As $\dom(f)$ is open, we find $\delta_0>0$ such that $B_{C([-h,0];\R^n)}[\phi_0,\delta_0]\subseteq \dom(f)$; let $\beta >\|\phi\|_{\textnormal{Lip}}$. As $f$ is almost locally Lipschitz, we find $0<\delta<\delta_0$ such that $f|_{D} \colon D\subseteq C([-h,0];\R^n)\to \R^m$ is Lipschitz continuous, where $D\coloneqq V_\beta\cap B_{C([-h,0];\R^n)}[\phi_0,\delta]$. Thus, using Theorem \ref{thm:set}, we estimate for all $\phi,\psi\in D$
\[
   \|f|_{D}(\phi)-f|_{D}(\psi)\|\leq \|f|_{D}\|_{\textnormal{Lip}}\|\phi-\psi\|_{C([-h,0];\R^n)]}\leq (h^{1/2}+h^{-1/2})\|f|_{D}\|_{\textnormal{Lip}}\|\phi-\psi\|_{H^1(-h,0;\R^n)}.
\]
Thus, $f|_D$ is Lipschitz continuous as mapping {from} $H^1(-h,0;\R^n)$. Since, by our preliminary observation, $D\subseteq H^1(-h,0;\R^n)$ is closed and convex, $f|_{D}$ admits a Lipschitz continuous extension, $F$, to the whole of $H^1(-h,0;\R^n)$, by Corollary  \ref{cor:Lipf}.
\end{proof}

The model in \cite{GW14} reads as follows
\[
   \begin{pmatrix}
       w(t) \\ v(t)
   \end{pmatrix}' = \begin{pmatrix} q(v(t))w(t) \\ -\mu v(t) + \frac{\gamma(v(t-r(v_{(t)})))g(x_2,v(t))w(t-r(v_{(t)}))}{g(x_1,v(t-r(v_t)))}\e^{\int_0^{r(v_{(t)})} d(y(s,v_{(t)}),v(t-s))\d s}\end{pmatrix},
\]
where $g,q,\gamma,d$ are suitable functions and $x_1<x_2$, $\mu$ are given real parameters. Assuming merely Lipschitz continuity and suitable boundedness requirements, Theorem \ref{thm:wpfde} {can be applied, which will be addressed in a future publication.} Here, we shall discuss the mapping $r$ in some more detail. {We will show that the assumptions given in \cite{GW14,BGR21} imply that $r$ is a Lipschitz continuous mapping on the whole of $H^1(-h,0)=H^1(-h,0;\R)$.}

{Next, following \cite{GW14,BGR21}, we will introduce the particulars necessary to define $r$:} We assume that $g\colon \R^2\to \R$ is Lipschitz continuous satisfying
\[
   \varepsilon \leq g(y,p)\leq K\quad(y,p\in\R)
\]
for some $0<\varepsilon<K$. Then consider for $\psi\in H^1(-h,0)$ the differential equation
\[
   y'(s)=-g(y(s),\psi(-s))\;(s>0),\quad y(0)=x_2.
\]
The time $r(\psi)\in [-h,0]$ is then given by $y(-r(\psi))=x_1$. Recall that the classical Picard--Lindel\"of theorem yields unique existence of $y$. Also, the assumptions on $g$ yield existence and uniqueness of  $r(\psi)\in [-h,0]$ if $h$ is large enough, see again, e.g.,~\cite{GW14}.

The rest of this section is devoted to establishing the following result, which confirms even global Lipschitz continuity of the delay functional $r$.
\begin{thm}\label{thm:lcon}
The mapping
\[
   r\colon H^1(-h,0)\to [-h,0]
\]
is Lipschitz continuous.
\end{thm}
\begin{proof}
  Let $\phi,\psi\in H^1(-h,0)$. We denote the corresponding solutions of the above ODEs determining $r(\phi)$ and $r(\psi)$ by $y_\phi$ and $y_\psi$, respectively. Then by integrating the ODEs in question, we obtain
  \[
     x_1 = x_2 -\int_0^{-r(\psi)} g(y_\psi(t),\psi(-t))\d t =x_2 -\int_0^{-r(\phi)} g(y_\phi(t),\phi(-t))\d t.
  \]
  Hence, from $r(\psi)<r(\phi)$ (without restriction) we deduce from the assumption on $g$,
  \begin{align*}
    \varepsilon |r(\phi)-r(\psi)| & \leq \int_{-r(\phi)}^{-r(\psi)} g(y_\phi(t),\phi(-t))\d t \\
    & = \int_0^{-r(\psi)} (g(y_\psi(t),\psi(-t))-g(y_\phi(t),\phi(-t)))\d t \\
    & \leq \int_0^h |(g(y_\psi(t),\psi(-t))-g(y_\phi(t),\phi(-t)))|\d t \\
        & \leq \int_0^h L|y_\psi(t)-y_\phi(t)|\d t+
        \int_0^h L|\psi(-t)-\phi(-t)|\d t,
  \end{align*}
  where $L$ is the Lipschitz constant of $g$. Using the differential equations for $y_\phi$ and $y_\psi$ and integrating, we get
  \[
    | y_\psi(t)-y_\phi(t)|\leq \int_0^t L|y_\psi(s)-y_\phi(s)|\d s+\int_0^t L|\psi(-s)-\phi(-s)|\d s.
  \]
  Gronwall's inequality thus implies
  \[
    | y_\psi(t)-y_\phi(t)|\leq \int_0^t L|\psi(-s)-\phi(-s)|\d s \e^{Lt}\leq L\sqrt{t}\e^{Lt}\|\phi-\psi\|_{L_2}
  \]
  Thus, combining the estimates, we get, for some number $C\geq 0$ depending on $h$ and $L$ only,
  \[
     \varepsilon |r(\phi)-r(\psi)|\leq C\|\phi-\psi\|_{L_2}\leq C\|\phi-\psi\|_{H^1}
  \]
   as required.
\end{proof}
\begin{rem}
Note that the estimate derived in the proof of Theorem \ref{thm:lcon} is better than anticipated in the claim of the theorem. {In fact, $r$ is merely Lipschitz continuous under the $L_2$-norm rather than the $H^1$-norm. However, in order to guarantee existence of $r(\phi)$ continuity of $\phi$ is required. Thus, the domain of $r$ is $H^1$, which is embedded into the space of continuous functions by Theorem \ref{thm:set}.}\end{rem}

\subsection{A positioning problem}\label{sec:applpos}

The second illustration of our results is taken from \cite{Walther} and describes a device positioned at the point $x(t)$ on the real axis determining its position by emitting signals to a given reference point $-w<0$, where the signal with constant speed $c$ is reflected. After some time, $s(t)$, depending on the position of the device, the signal is received again by the device. The time $s(t)$ yields a position $\hat{x}$ according to 
\[
   \hat{x}=\tfrac{c}{2} s(t)-w,
\]
this derived position being the actual position if $x(t)=0$. For the right-hand side this computed position is used for the acceleration $a(\hat{x})$. As in \cite[Section 3]{Walther}, we consider the given model only if $x$ is in between $-w$ and some a priori given $w_+>0$. For some Lipschitz continuous $a\colon \R\to \R$, the set of equations we consider here is
\[
   \begin{pmatrix}x\\v \end{pmatrix}' =    \begin{pmatrix}v\\-\mu v + a(s\tfrac{c}{2}-w)  \end{pmatrix}\text{ with }cs(t)=x(t-s(t))+x(t)+2w. 
\]for some fixed real numbers $c,\mu$. 

As in \cite{Walther}, we let $h\coloneqq (2w+2w_+)/c$ and obtain  $0<s\leq h$ assuming $x(t)<w_+$. Thus, $h$ is our final time horizon (in negative direction). Next, we {want to} apply our main theorem (Theorem \ref{thm:wp}) and the permanence principle (Theorem \ref{thm:pp}) for obtaining a solution theory for all initial values $\phi\in H^1(-h,0;\R^2)$ with bounded derivative. {In order to provide the proper setting, w}e define for $0<\alpha<\min\{c,w,w_+\}$
\[
    W_\alpha \coloneqq \{ \phi=(\xi,\eta)\in H^1(-h,0;\R^2); -w+\alpha\leq \xi\leq w_+-\alpha, |\xi'(s)|\leq c-\alpha\text{ for a.e.~}s\in (-h,0)\}.
\]
It is not difficult to see that $W_\alpha$ is a closed and convex subset of $H^1(-h,0;\R^2)$.

\begin{thm}\label{thm:dfw} Let $\phi=(\xi,\eta)\in W_\alpha$. Then there exists a unique $s_\phi\in [0,h]$ such that
\[
   cs_\phi = \xi(-s_\phi)+\xi(0)+2w.
\]
Moreover, 
\[
   W_\alpha \ni \phi\mapsto s_\phi\in [0,h]
\]
is Lipschitz continuous.  
\end{thm}
\begin{proof}
 Let $s,t\in [-h,0]$, $s<t$. Then by the fundamental theorem of calculus for $H^1$-functions, we obtain for $\phi=(\xi,\eta)\in W_\alpha$
 \[
    |\xi(s)-\xi(t)|\leq \int_{s}^t |\xi'(\tau)|\d \tau \leq (c-\alpha)|t-s|.
 \]
 Thus, $[0,h]\ni s\mapsto \tfrac{1}{c}(\xi(-s)+\xi(0)+2w)$ is a self-mapping strict contraction and thus admits a unique fixed point. 
 Similarly to the proof of the same property in \cite{Walther}, we estimate for $\phi=(\xi,\eta),\psi=(\zeta,\chi)\in W_\alpha$
 \begin{align*}
   c|s_\phi-s_\psi|& =|\xi(-s_\phi)-\zeta(-s_\psi)+\xi(0)-\zeta(0)| \\
   & \leq (c-\alpha)|s_\phi-s_\psi|+2 \sup_{s\in [-h,0]}|\phi(s)-\psi(s)|.
 \end{align*}
 Hence, by Theorem \ref{thm:set}, 
 \[
    |s_\phi-s_\psi|\leq \frac{2}{\alpha}(h^{1/2}+h^{-1/2})\|\phi-\psi\|_{H^1(-h,0;\R^2)}.\qedhere
 \]
\end{proof}
Next, with the projection $\pi_\alpha\colon H^1(-h,0;\R^2)\to H^1(-h,0;\R^2)$ on $W_\alpha$ and 
\[r_\alpha\colon H^1(-h,0;\R^2)\to [-h,0]\] given by $r_\alpha(\phi) = s_{\pi_\alpha(\phi)}$, the assumptions of Theorem \ref{thm:wp} are satisfied, for the right-hand side
\begin{multline*}
G(t,(x,v)_t)\coloneqq g(t, (x(t),v(t)), (x(t+r_\alpha((x,v)_{t})),v(t+r_\alpha((x,v)_{t})))) \\ = \begin{pmatrix}v(t)\\-\mu v(t) + a( \frac{x(t+r_\alpha((x,v)_{t}))+x(t)}{2}) \end{pmatrix}.
\end{multline*}
{Indeed, (H1) is evident by the Lipschitz continuity of $a$; note that in \cite{Walther} following up on a solution theory in $C^1$, $a$ needed to be continuously differentiable. For the satisfaction of (H2), we need to find a Lipschitz estimate for the delay functional in terms of $x$ on $V_\beta$. This is provided in Theorem \ref{thm:dfw}, which is a Sobolev regular variant of the similar result \cite[Proposition 8]{Walther}. We emphasise that we do not need a differentiability statement there.}
{By Theorem \ref{thm:dfw} for a Lipschitz continuous $\phi \in H^1(-h,0;\R^2)$ we obtain unique existence of $(x,v)$ such that
\[
   \begin{pmatrix} x\\ v\end{pmatrix}'=G(\cdot,(x,v)_{(\cdot)}) \text{ subject to } (x,v)_0 =\phi
\]
If, in addition, $\phi\in W_{\alpha'}$ for some $0<\alpha'<\alpha$, then by the permanence principle Theorem \ref{thm:pp}, we deduce the existence of some $T>0$ such that $(x,v)_t \in W_\alpha$ for all $0\leq t\leq T$ and, thus, $(x,v)$ satisfies on $(0,T)$
\[
  \begin{pmatrix} x\\ v\end{pmatrix}'=\begin{pmatrix}v(\cdot)\\-\mu v(\cdot) + a\big( \frac{x(\cdot+s_{(x,v)_{(\cdot)}})+x(\cdot)}{2}\big) \end{pmatrix}.
\] }
\section{Further comments on other approaches concerning functional differential equations}\label{sec:FDECW}
{
In this section, we fix $h,T>0$ and consider a functional differential equation {in a simplified form. For a scalar state $x$ consider}
\begin{equation}\label{eq:fdef}
   x'(t) = f(x_t), t\in [0,T],
\end{equation}
subject to some initial pre-history $\phi\colon [-h,0]\to \R$, where $f\colon U\to \R$ with $f$ defined on a subset $U$ of a function space of functions from $[-h,0]$ into $\R$.}

(a) Let us point out the subtle difference of the functional differential equation in \cite{BGR21,MNP94} and the one in Theorem \ref{thm:fdewpint}. In \cite{MNP94}, the existence of solutions was shown for continuous $G\colon C[-h,0]\to\R$; uniqueness requires almost local Lipschitz continuity for $G$. The delay functional, $r$, considered in Theorem \ref{thm:dfw} is defined on Lipschitz continuous functions satisfying a certain bound for the functions itself and for the Lipschitz norm. Recall $V_\alpha\coloneqq \{x\in H^1(-h,0); |x'(t)|\leq \alpha\text{ for a.e.~}t\in (-h,0)\}$. In order to apply the retraction approach in {the setting \eqref{eq:fdef}, where the continuity properties of the delay functional $r$ are inherited by a suitable $f$}, one is required to analyse, whether there exists a continuous extension of $f\colon V_\alpha\subseteq C[-h,0]\to \R$ to the whole of $C[-h,0]$ given $f$ is continuous. Since $V_\alpha$ is a closed and convex subset of $C[-h,0]$ it is possible to apply \cite[Proposition 1.19(i)]{BL00} to obtain the \textbf{retraction} $p\colon C[-h,0]\to V_\alpha$ (i.e., a continuous map such that $\rho|_{V_\alpha}=\id_{V_\alpha}$). Thus, $F\coloneqq f\circ p$ extends $f$ continuously to the whole of $C[-h,0]$. Whether this extension retains to be almost locally Lipschitz if $f$ has the said property is a priori unclear. Note that for this, by \cite[Lemma 2.7]{BGR21}, it suffices to show that $p$ can be chosen to be (locally) Lipschitz. The quest is to show whether or not $p$ can be chosen to be Lipschitz continuous. 

In the following we describe possible arguments that are relevant to this problem. 

(b) An argument in \cite{MNP94} shows that $C([-h,0];[A,B])$ for some $-\infty<A<B<\infty$ is a Lipschitz \textbf{retract}, that is, there exists a Lipschitz continuous retraction from $C[-h,0]$ onto $C([-h,0];[A,B])$. Next, McShane's Lemma shows $W_{1,0}\coloneqq \{f\in C[-h,0]; \|f\|_{\text{Lip}}\leq 1, f(0)=0\}$ is a Lipschitz retraction of $C[-h,0]$, see \cite{M34}. Indeed, for this, consider the mapping
\[
  E \colon C[-h,0]\to C[-h,0]\cap V_1, f\mapsto\big(s\mapsto \min_{t\in [-h,0]} (f(t)-|s-t|)\big),
\]
which is a Lipschitz retraction on the space of Lipschitz continuous functions with Lipschitz semi-norm bounded by $1$. The mapping $H\colon C[-h,0]\cap V_1\to W_{1,0}, f\mapsto f-f(0)$ is, in turn, Lipschitz continuous. Then $E\circ H$ is the desired retraction.

(c) Whether or not $W_\alpha$ from above is a Lipschitz retract of $C([-h,0];\R^n)$ does not follow from (b) and needs to be addressed separately. More generally,  let $W\subseteq C[-h,0]$ be closed and convex. Note that by the Sobolev embedding theorem, Theorem \ref{thm:set}, any closed and convex subset  $W\subseteq C[-h,0]$ yields a closed and convex subset $W\cap H^1(-h,0)\subseteq H^1(-h,0)$. Hence, $W\cap H^1(-h,0)$ is a Lipschitz retract by Proposition \ref{prop:normproj} and any Lipschitz continuous mapping on $G\colon W\cap H^1(-h,0)\subseteq C[-h,0]\to \R$ can be extended to a Lipschitz continuous mapping on the whole of $H^1(-h,0).$

(d) By \cite[Theorem 1.6]{BL00}, $C[-h,0]$ is an absolute Lipschitz retract, that is, for every metric space $M$ containing $C[-h,0]$, there exists a Lipschitz continuous retraction from $M$ onto $C[-h,0]$. $W$ from (c) is a Lipschitz retract of $C[-h,0]$ if and only if $W$ is a so-called absolute Lipschitz retract, see \cite[Proposition 1.2 or subsequent remark (ii)]{BL00}. In fact, this puts one into the situation of \cite{HR23}, where it was shown that if $W$ is additionally a compact subset of $C[-h,0]$ then such sets are so-called absolute uniform retracts, still a weaker property than being an absolute Lipschitz retract. Whether or not (a closed, compact) $W$ is an absolute Lipschitz retract apparently needs to be addressed  for each set  {separately} as there does not seem to be a general answer, see in particular the counterexample in \cite[Section 5]{HR23}.

\section{Conclusion}\label{sec:con}

We have developed a solution theory for state-dependent delay equations (and in passing) for a certain class of functional differential equations that require Lipschitz continuity assumptions only. The assumptions are easily verified in practice and do not involve potentially complicated computations and analysis of certain derivatives. The underlying Hilbert space structure permits localisation techniques using the projection theorem so that many equations can be dealt with on the whole domain. The concept of a solution manifold is not needed for existence and uniqueness of solutions to differential equations with state-dependent delay. Being an element of the solution manifold for the initial value is a mere requirement for {continuously differentiable} solutions. Solutions starting off of the solution manifold will eventually end up on the solution manifold if the solution exists for a long enough time. Other approaches not using the solution manifold are based on compactness results, Schauder's fixed point theorem and a retraction technique. {Compactness results or Schauder's fixed point theorem are not required in the present approach. Moreover, by Kirszbraun's Theorem, see \cite{K34}, a retraction approach seems to be more widely applicable in the Hilbert space framework developed in this manuscript.}

\section*{Competing Interests Statement}

There are no competing interests to declare.


\begin{thebibliography}{10}

\bibitem{AF03}
R.~Adams and J.~Fournier.
\newblock{\em Sobolev spaces.} 2nd ed.
\newblock Pure and Applied Mathematics 140. New York, NY: Academic Press, 2003.

\bibitem{ISem18}
W.~Arendt, R.~Chill, C.~Seifert, H.~Vogt, J.~Voigt
\newblock{\em Form Methods for Evolution Equations, and Applications}
\newblock Lecture Notes, 18th Internet Seminar
\newblock \url{https://www.mat.tuhh.de/veranstaltungen/isem18/pdf/LectureNotes.pdf}


\bibitem{AFK10}
D. Azagra, R. Fry, and L. Keener.
\newblock Smooth extensions of functions on separable Banach spaces. 
\newblock{\em Math. Ann.} 347(2), 285--297, 2010.
and Erratum to: Smooth extensions of functions on separable Banach spaces D. Azagra, R. Fry \& L. Keener, 
Math. Ann. 350, 497--500, 2011.

\bibitem{BGR21}
I. Bal{\'a}zs, P. Getto and G. R\"ost.
\newblock{A continuous semiflow on a space of Lipschitz functions for a differential equation with state-dependent delay from cell population biology}.
\newblock{\em Journal of Differential Equations}, 304:73 -- 101, 2021.

\bibitem{BL00}
Y.~Benyamini, J.~Lindenstrauss
\newblock{\em Geometric nonlinear functional analysis. Volume 1.}
\newblock Colloquium Publications. American Mathematical Society (AMS). 48. Providence, RI: American Mathematical Society (AMS), 2000.


\bibitem{BK15}
M.~Brokate and G.~Kersting.
\newblock{\em Measure and integral.}
\newblock Compact Textbooks in Mathematics. Cham: Birkh\"auser/Springer, 2015.

\bibitem{D95}
O. Diekmann, S.~van Gils, S.~Verduyn Lunel, H.O.~Walther.
\newblock{\em Delay equations. Functional-, complex-, and nonlinear analysis.}
\newblock Applied Mathematical Sciences. 110. New York, NY: Springer-Verlag, 1995.


\bibitem{EG15}
L.C.~Evans and R.F.~Gariepy
\newblock{ \em Measure theory and fine properties of functions.}
\newblock Textbooks in Mathematics. Boca Raton, FL: CRC Press, 2015.


\bibitem{GW14}
P. Getto and M. Waurick. 
\newblock{A differential equation with state-dependent delay from cell population biology}.
\newblock{\em Journal of Differential Equations}, 260:6176--6200, 2014.

\bibitem{HJ10}
P.~H\'ajek and M.~Johanis.
\newblock{ On Peano's theorem in Banach spaces.}
\newblock{\em J. Differ. Equations} 249, No. 12, 3342--3351, 2010.

\bibitem{HR23}
P.~H\'ajek and R.~Medina
\newblock{Compact retractions and Schauder decompositions in Banach spaces}
 \newblock{\em Trans. Amer. Math. Soc.} 376, 1343--1372, 2023.

\bibitem{HV93}
J.~Hale and S.~Verduyn Lunel
\newblock{\em Introduction to functional differential equations.} 
Applied Mathematical Sciences. 99. New York, NY: Springer-Verlag, 1993.
 
% \bibitem{Hp23}
% P.~H\'ajek
% \newblock Personal communication. 2023.

\bibitem{HKWW06}
F. Hartung, T. Krisztin, H-O. Walther and J. Wu. 
\newblock{Chapter 5 - Functional Differential Equations with State-Dependent Delays: Theory and Applications}. 
\newblock{\em Handbook of Differential Equations: Ordinary Differential Equations, Volume 3}, 435--545, 2006.

\bibitem{HT97}
F. Hartung and J. Turi, 
\newblock{On differentiability of solutions with respect to parameters in state-dependent delay equations}, 
\newblock{\em J. Differential Equations}, 135:192--237, 1997.

\bibitem{K34}
M.D.~Kirszbraun, 
\newblock{\"Uber die zusammenziehende und Lipschitzsche Transformationen.} 
\newblock{\em Fundamenta Mathematicae} 22: 77--108, 1934.

  \bibitem{Drakalauch}
  A.~Kalauch, S. Siegmund, R. Picard, S. Trostorff, and M. Waurick
\newblock{A Hilbert Space Perspective on Ordinary Differential Equations with Memory Term }
\newblock{\em Journal of Dynamics and Differential Equations.} 26(2):369--399, 2014.


\bibitem{MC16}
Y. Li, M. Mahdavi, and C. Corduneanu.
\textit{Functional Differential Equations: Advances and Applications}. 
John Wiley \& Sons, Incorporated, 2016.

\bibitem{MNP94}
J. Mallet-Paret,~R.D. Nussbaum,~P. Paraskevopoulos,
\newblock{ Periodic solutions for functional differential equations with
multiple state-dependent time lags}, 
\newblock{\em Topol. Methods Nonlinear Anal.} 3, 101--162, 1994.

\bibitem{M34}
E.~McShane,
\newblock{Extension of range of functions.}
\newblock{\em Bull. Am. Math. Soc.} 40, 837-842, 1934.

\bibitem{M98}
R.~E.~Megginson,
\newblock {\em An introduction to Banach space theory}
Graduate Texts in Mathematics. 183. New York, NY: Springer, 1998.

\bibitem{N17}
J.~Nishiguchi,
\newblock{A necessary and sufficient condition for well-posedness of initial value problems of retarded functional differential equations}
\newblock{\em Journal of Differential Equations}, 263,  3491--3532, 2017.

\bibitem{N18}
J.~Nishiguchi,
\newblock{Theory of well-posedness for delay differential equations via prolongations and $C^1$-prolongations: its application to state-dependent delay}, 
\newblock{arXiv:1810.05890}



\bibitem{Picard2009}
R.~Picard
\newblock {A structural observation for linear material laws in classical mathematical physics},
\newblock {\em Mathematical Methods in the Applied Sciences}, 32(14):1768--1803,2009.

\bibitem{PTW14}
R.~Picard, S.~Trostorff, and M.~Waurick
\newblock{A Functional Analytic Perspective to Delay Differential Equations }
\newblock{ \em Oper. Matrices} 8(1): 217-236, 2014.

%\bibitem{FMS23}
%F.~M.~Schneider,
%\newblock Personal communication, 2023.

\bibitem{STW22}
C.~Seifert, S.~Trostorff, and M.~Waurick
\newblock {\em Evolutionary Equations}
\newblock {Birkh\"auser, Cham, 2022}

\bibitem{Walther}
H-O. Walther. 
\newblock{The solution manifold and \(C^1\)-smoothness for differential equations with state-dependent delay}. 
\newblock{\em Journal of Differential Equations}, 195:46--65, 2003. 

\bibitem{W04}
H.-O. Walther,
\newblock{ Smoothness Properties of Semiflows for Differential Equations with State-Dependent Delays
}
\newblock{\em Journal of Mathematical Sciences} 124, 5193--5207, 2004.

\bibitem{Walther2}
H.O. Walther, 
\newblock Evolution systems for differential equations with variable time lags, 
\newblock{\em J. Mathematical Sciences} 202, 911--933, 2014.



\bibitem{W70}
E. Winston, 
\newblock{Uniqueness of the zero solution for differential equations with state-dependence}
\newblock{\em  J. Differential Equations} 7, 395--405, 1970.




\end{thebibliography}
\end{document}